\documentclass[11pt,reqno]{amsart}

\usepackage[utf8]{inputenc}

\usepackage[T1]{fontenc}
\usepackage{lmodern}

\usepackage[italian,english]{babel}

\usepackage[margin=1in]{geometry} 

\usepackage{graphicx} 
\usepackage{float} 

\usepackage[parfill]{parskip} 
 
\usepackage{booktabs} 
\usepackage{array}
\usepackage{paralist} 
\usepackage{verbatim} 
\usepackage{subfig} 
\usepackage{mathrsfs}
\usepackage{amsmath}
\usepackage{amssymb}
\usepackage{amsfonts,amssymb,esint}
\usepackage{amsthm}
\usepackage{graphics,color}
\usepackage{enumerate, enumitem}
\usepackage{mathtools,centernot}
\usepackage{cases}
\usepackage{amsrefs}

\usepackage{tikz}

\usepackage{bbm}
\usepackage{xfrac}
\usepackage{hyperref}
\usepackage[noabbrev, capitalize]{cleveref}
\usepackage{autonum}

\usepackage{bookmark}

\newtheorem{theorem}{Theorem}[section]
\newtheorem{lemma}[theorem]{Lemma}

\newtheorem{corollary}[theorem]{Corollary}
\newtheorem{definition}[theorem]{Definition}
\newtheorem{proposition}[theorem]{Proposition}

\numberwithin{equation}{section} 

\newcommand{\abs}[1]{\left|#1\right|}

\newcommand{\norm}[1]{\left\|#1\right\|}

\newcommand*{\R}{\ensuremath{\mathbb{R}}}

\newcommand*{\N}{\ensuremath{\mathbb{N}}}

\newcommand{\eps}{\varepsilon}

\newcommand*{\G}{\ensuremath{\mathcal{G}}}
\newcommand*{\F}{\ensuremath{\mathcal{F}}}
\newcommand*{\cN}{\ensuremath{\mathcal{N}}}

\newcommand{\e}{\varepsilon}

\newcommand{\Per}{\textnormal{Per}}

\renewcommand{\MR}[1]{} 

\usepackage{color, graphicx}
\usepackage{mathrsfs, dsfont}

\usepackage[]{hyperref}
\hypersetup{
    colorlinks=true,       
    linkcolor=red,          
    citecolor=blue,        
    filecolor=red,      
    urlcolor=cyan           
}

\def\dist{\mathop{\rm d}\nolimits}    
 
\def\sgn{\mathop{\rm sgn\,}\nolimits}    

\def\d{\mathop{\rm d}\nolimits} 

\newcommand{\be}{\begin{equation}}
\newcommand{\ee}{\end{equation}}

\title{\mbox{$\Gamma$-convergence for higher order nonlocal phase transitions}}

\author{Hardy Chan}
\address{Departement Mathematik Und Informatik, Universit\"at Basel, Basel, Switzerland}
\email{hardy.chan@unibas.ch}

\author{Serena Dipierro}
\address{Department of Mathematics and Statistics, University of Western Australia, Perth, Australia}
\email{serena.dipierro@uwa.edu.au}

\author{Mattia Freguglia}
\address{Department of Decision Sciences, Bocconi University, Milano, Italy}
\email{mattia.freguglia@unibocconi.it}

\author[Marco Inversi]{Marco Inversi}
\address{Max Planck Institute for Mathematics in the Sciences, Leipzig, Germany.}
\email{marco.inversi@mis.mpg.de}

\author{Enrico Valdinoci}
\address{Department of Mathematics and Statistics, University of Western Australia, Perth, Australia}
\email{enrico.valdinoci@uwa.edu.au}

\date{\today}



\begin{document}

\begin{abstract}

For every $0 < s <\sfrac{3}{4}$, we study the asymptotic behavior of the $\eps$-rescaled sum of the $s$-fractional Allen-Cahn energy and the squared $L^2$-norm of its first variation. We prove that the contribution of the first variation vanishes as $\eps \to 0$. This implies the Gamma-convergence of the initial sum to either the classical perimeter or to the $2s$-fractional perimeter, depending on whether $s \ge \sfrac{1}{2}$ or not. This contradicts the expectation of finding curvature-dependent terms in the limit, as suggested by the regime $\sfrac{3}{4} \le s < 1$, and as known to hold in low dimensions in the local case.

\end{abstract}

\maketitle

\section{Introduction} \label{s: introduction}

For every $s \in (0,1)$ and $\eps > 0$, Savin and the fifth author~\cite{SV12} studied the following families of fractional Allen--Cahn energies  
\begin{equation} \label{eq:AC-fract}
        \mathcal{F}_{s,\eps}(u, \Omega)
        = \alpha_s(\eps)\left[\frac{\gamma_{d,s}}{4} \iint_{\R^d \times \R^d \setminus (\Omega^c \times \Omega^c)}
        \frac{\abs{u(x)-u(y)}^2}{\abs{x-y}^{d+2s}} \, dx dy + \int_{\Omega} \frac{W(u(x))}{\eps^{2s}} \, dx \right].  
    \end{equation}
Here, $\Omega \subset \R^d$ is a bounded Lipschitz domain, $W \colon \R \to [0,+\infty)$ is a double-well potential, e.g. $W(u)=(1-u^2)^2$ (see Section~\ref{s: tools} for the precise assumptions), and $\alpha_s(\eps)$ is a scaling parameter defined by $\alpha_s(\eps) = \eps^{2s-1}$ if $s \in \left(\sfrac{1}{2},1\right)$, $\alpha_{\sfrac{1}{2}}(\eps) = \abs{\log(\eps)}^{-1}$, and $\alpha_s(\eps) = 1$ if $s \in \left(0,\sfrac{1}{2}\right)$. Moreover, $\gamma_{d,s}$ is a normalization constant (see~\eqref{eq: constant fractional laplacian} for its explicit value) ensuring that $\F_{s,\eps}(u, \Omega)$ converges to the classical Allen--Cahn energy 
\begin{equation}
        \label{eq.local-AC}
            \mathcal{F}_{1,\eps}(u, \Omega) = 
            \int_{\Omega} \eps \frac{\abs{\nabla u}^2}{2} \, dx + \int_{\Omega} \frac{W(u)}{\eps} \, dx
    \end{equation}
in the limit as $s \to 1$ (see~\cite{BBM01}). A celebrated result by Modica and Mortola~\cite{MM77} asserts that the local Allen--Cahn energies $\{\F_{1,\e}\}$ approximate the perimeter functional in the sense of Gamma-convergence, in the limit as $\e \to 0$. The latter is a notion of variational convergence for functionals, introduced in~\cite{DG-Fra} (see also the monograph~\cite{DalMaBook93}), well-suited for studying the asymptotic behavior of minima and minimizers. Moving to the nonlocal regime, the authors of~\cite{SV12} computed the $\Gamma$-limit of the family $\{\mathcal{F}_{s,\eps}\}$, as $\eps \to 0$, and observed different behaviors depending on the value of the parameter $s \in (0,1)$. More precisely, they proved that 
\begin{equation} \label{eq: SV full}
    \Gamma-\lim_{\eps \to 0^+} \mathcal{F}_{s,\eps}(\chi_E, \Omega) = \begin{cases}
        c_{\star} \mathrm{Per}(E, \Omega) & s \ge \sfrac{1}{2}, 
        \\[0.5ex] \gamma_{d,s} \mathrm{Per}_{2s}(E,\Omega) & s < \sfrac{1}{2}.
    \end{cases}
\end{equation}
Here, the $\Gamma$-limit is taken with respect to the local $L^1$-topology in $\R^d$, $\chi_E(x) = 1$ if $x \in E$ and $\chi_E(x) = -1$ if $x \notin E$,  $c_{\star} > 0$ is a constant independent of the set $E$, $\mathrm{Per}(E, \Omega)$ and $\mathrm{Per}_{2s}(E, \Omega)$ denote the perimeter and the $2s$-fractional perimeter of the set $E$ in $\Omega$ respectively. For $s \geq \sfrac{1}{2}$ this result is consistent with the Modica--Mortola Theorem in the local case, while for $s < \sfrac{1}{2}$ the nonlocal behavior persists in the limit. 

This paper is aimed to the analysis of the asymptotic behavior of the sum of $\mathcal{F}_{s,\eps}$ and the rescaled squared $L^2$-norm of its first variation. More precisely, we define
\begin{equation}
        \label{eq:frac-Bel-Pao}
        \G_{s,\e}(u, \Omega) : = \beta_s(\eps)
        \int_{\Omega} \bigg( (-\Delta)^s u + \frac{W'(u)}{\e^{2s}} \bigg)^2 \, dx, \qquad
        \beta_s(\eps) :=
            \begin{cases}
                \eps^{4s-3} & \text{if $s \in (\sfrac{3}{4},1]$}, \\
                \abs{\log(\eps)}^{-1} & \text{if $s = \sfrac{3}{4}$}, \\
                1 & \text{if $s \in (0,\sfrac{3}{4})$}.
            \end{cases}
    \end{equation}
Here, $(-\Delta)^s u + W'(u) \eps^{-2s}$ represents the first variation of $\mathcal{F}_{s,\eps}$, and $(-\Delta)^s u$ denotes the $s$-fractional Laplacian of the function $u \colon \R^d \to \R$ (see~Section~\ref{s: tools} for the definition). 

We first discuss the case $s \in [\sfrac{1}{2},1]$. In this situation, it would be natural to expect that the asymptotic behavior of the functionals $\G_{s,\e}$ involves the mean curvature, which represents the first variation of the $\Gamma$-limit of $\mathcal{F}_{s,\eps}$. Specifically, we are interested in the following question: 

\begin{center}
\textit{Fix $s \in [ \sfrac{1}{2}, 1 ]$, let $E \subset \R^d$ be a smooth bounded open set, and let $u_\e \to \chi_E$ be a family of functions such that $\F_{s,\e} (u_\e, \Omega)$ is uniformly bounded. Does there exist a positive constant $\kappa_{\star} > 0$, independent of the set $E$, such that}
\end{center}
\begin{equation}
        \label{eq:liminf-intro}
        \liminf_{\eps \to 0^+} \G_{s,\eps}(u_\eps, \Omega) \ge \kappa_{\star} \int_{\partial E \cap \Omega} H_{\partial E}(y)^2 \, d\mathcal{H}^{d-1}(y) ?
\end{equation}

In the above formula, $H_{\partial E}(y)$ denotes the mean curvature of $\partial E$ at the point $y$, and $\mathcal{H}^{d-1}$ stands for the $(d-1)$-dimensional Hausdorff measure in $\R^d$. The functional on the right-hand side of~\eqref{eq:liminf-intro} is the geometric energy known as the Willmore functional~\cites{Will93, KS12} (we refer to Remark (a) for further comments about this energy).

For $s < \sfrac{1}{2}$, a similar problem can be investigated. For example, one may study the validity of an inequality of the form of~\eqref{eq:liminf-intro} with some nonlocal quantity on the right hand side, such as the $2s$-fractional mean curvature~\cites{AbVal14, FFMMM15}.  

We point out that, by the Gamma-convergence results in~\cite{Modica} (for $s=1$) and in~\cite{SV12} (for $\sfrac{1}{2} \le s <  1$) it follows that~\eqref{eq:liminf-intro} holds whenever $\{ u_\eps \}$ is a family of minimizers of the functional $\mathcal{F}_{s,\eps}$. In this case, $E$ turns out to be a perimeter-minimizing set in $\Omega$, and hence its mean curvature vanishes.

\subsection{A brief background} 
The local case $s=1$ has been widely investigated, since it is closely related to a conjecture of De Giorgi~\cite{DeGiorgi}*{Conjecture~4} (see also Remark (c) below for further details). In particular, Bellettini and Paolini~\cite{Bel-Pao93} observed that
    \begin{equation}
        \label{eq:limsup-local-1}
            \lim_{\eps \to 0^+} \G_{1,\eps}(u_\eps, \R^d) = \sigma_W \int_{\partial E} H_{\partial E}(y)^2 \, d\mathcal{H}^{d-1}(y), \qquad \sigma_W := \int_{-1}^1 \sqrt{2 W(s)} \, ds,
    \end{equation}
where $u_\eps \to \chi_E$ is the usual recovery sequence for the functionals $\F_{1,\eps}$. More precisely, $u_\eps$ is obtained via a slight modification of the function $q(\d_{\partial E}(x) / \eps)$, where $q$ is the unique solution of $\dot{q} = \sqrt{W(q)}$ with $q(0)=0$, and $\d_{\partial E}$ denotes the signed distance function from $\partial E$. However, establishing the corresponding \emph{liminf inequality} (that is~\eqref{eq:liminf-intro} with $s=1$) for an arbitrary family of functions $u_{\eps} \to \chi_E$ turned out to be much more challenging. This has been proved in dimension $n \in \{2,3\}$ by Röger and Schätzle~\cite{Roger-Schatzle} (see also~\cite{Nagase-Tonegawa}), while it remains open in higher dimensions.

The case $s \in \left[\sfrac{3}{4},1 \right)$ has been recently considered by the first, third and fourth authors in~\cite{CFI24}. Specifically, they proved the counterpart of~\eqref{eq:limsup-local-1}, that is
    \begin{equation}
        \label{eq:limsup-nonlocal-1}
            \lim_{\eps \to 0^+} \G_{s,\eps}(u_\eps, \Omega) = \kappa_{\star} \int_{\partial E \cap \Omega} H_{\partial E}(y)^2 \, d\mathcal{H}^{d-1}(y), \qquad \kappa_{\star} > 0,
    \end{equation}
where $u_\eps \to \chi_E$ is a family of functions constructed from $w(\d_{\partial E}(x) / \eps)$, and $\Omega$ is a regular bounded domain. As in the local case, $w \colon \R \to (-1,1)$ denotes the one-dimensional optimal profile, which is the unique (up to translations) monotone increasing solution of the $s$-fractional Allen--Cahn equation in dimension one
    \begin{equation}
        \label{eq:optimal-profile-eq}
            (-\partial_{zz})^s w + W'(w) = 0.
    \end{equation}
The identity~\eqref{eq:limsup-nonlocal-1} mildly suggests that, for $s \in [\sfrac{3}{4},1)$, the inequality~\eqref{eq:liminf-intro} might hold, at least in low dimensions, as it occurs in the local case. Nevertheless, when $s \in \left[\sfrac{1}{2},1\right)$, the validity of this inequality remains unclear, even for families $\{ u_\eps \}$ of critical points of $\mathcal{F}_{s,\eps}$ with equibounded energy. As in the local case, it would be expected that the energies of the critical points concentrate on a ``generalized'' critical point of the perimeter.  

\subsection{Main result}
In this paper, we address the case $s \in \left( 0, \sfrac{3}{4}\right)$ and we establish the following Gamma-convergence result. For every $s \in [\sfrac{1}{2},\sfrac{3}{4})$, we prove that the inequality~\eqref{eq:liminf-intro} cannot hold with a positive constant. A similar phenomenon is shown when $s \in ( 0, \sfrac{1}{2})$. 

\begin{theorem} \label{t:main} 
Let $\Omega \subset \R^d$ be a bounded Lipschitz domain, and let $E \subset \R^d$ be a Borel set. 
\begin{enumerate}
    \item For every $s \in \left[\sfrac{1}{2},\sfrac{3}{4} \right)$, it holds that
\begin{equation}
        \Gamma - \lim_{\e\to 0^+} \big( \mathcal{F}_{s,\eps} + \G_{s,\e} \big) (\chi_E ,\Omega) = c_\star \Per(E, \Omega),  
\end{equation}
where $c_\star$ is the same constant appearing in~\eqref{eq: SV full}. 
\item For every $s \in \left(0,\sfrac{1}{2}\right)$, it holds that
    \begin{equation}
        \Gamma - \lim_{\e\to 0^+} \big( \mathcal{F}_{s,\eps} + \G_{s,\e} \big) (\chi_E ,\Omega) = \gamma_{d,s} \Per_{2s}(E, \Omega),  
    \end{equation}
where $\gamma_{d,s}$ is the constant defined by~\eqref{eq: constant fractional laplacian}. 
\end{enumerate}
In both cases, the Gamma-limit is computed with respect to the local $L^1$-topology in $\R^d$.
\end{theorem}

We observe that the scaling in the definition of $\G_{s,\eps}$ is the natural one to detect a potentially nontrivial Gamma-limit. For any $s \in \left(0,\sfrac{3}{4}\right)$, combining identity~\eqref{eq:intro-A} with Corollary~\ref{cor:N vanishes-2}, we~have
\begin{equation}
        \label{eq:optimal-scaling}
            \lim_{\eps \to 0^+} \G_{s,\eps}(u_\eps, \Omega) > 0,
    \end{equation}
where $u_\eps(x)=w((1-\abs{x})/\eps)$ is the standard recovery sequence for the unit ball $B \subset \R^d$, with respect to the functionals $\mathcal{F}_{s,\eps}$, and $\Omega$ is any open set intersecting $\partial B$. Surprisingly, the functionals $\G_{s,\e}$ do not produce curvature-dependent contributions in the limit.

\subsection{Heuristic of the proof}

We begin with the case $s\in \left[ \sfrac{1}{2}, \sfrac{3}{4}\right)$. In view of the Gamma-convergence result in~\cite{SV12}, the proof of \cref{t:main} is achieved if we exhibit a family of functions $u_\eps \to \chi_E$ in $L^1_{loc}(\R^d)$ such that
    \begin{equation}
        \label{eq:remark-a}
        \lim_{\eps \to 0^+} \mathcal{F}_{s,\eps}(u_\eps, \Omega) = c_\star \Per(E, \Omega), \quad \text{and} \quad \lim_{\eps \to 0^+} \mathcal{G}_{s,\eps}(u_\eps, \Omega) = 0.
    \end{equation} 
To simplify the presentation, we consider the case $\Omega=\mathcal{U}$, where $\mathcal{U}$ denotes a small tubular neighborhood of $\partial E$. In this setting, the ``model'' function $u_\eps(x):=w(\d_{\partial E}(x) / \eps)$ has the same regularity as the boundary of $E$, which we may assume to be smooth by density. In the general case, one needs to slightly modify $u_\eps$, so that it remains regular everywhere in $\Omega$, thereby ensuring that its fractional Laplacian is pointwise well-defined. As the zero level set of $\d_{\partial E}$ coincides with $\partial E$, it follows that the function $u_\eps$ converges smoothly to $\chi_E$ locally away from $\partial E$. This implies that
    \begin{equation}
        \label{eq:conv-lap}
        (-\Delta)^s u_\eps \longrightarrow (-\Delta)^s \chi_E, \quad \text{locally uniformly in $\mathcal{U} \setminus \partial E$}. 
    \end{equation}
Furthermore, from~\cite{DPV15}*{Proposition~7.2} one can derive that
    \begin{equation}
        \label{eq:conv-pot}
        \frac{W' ( u_\eps)}{\eps^{2s}} \longrightarrow - \frac{\gamma_{1,s}}{s} \frac{\chi_{E}}{\abs{\d_{\partial E}}^{2s}}, \quad \text{locally uniformly in $\mathcal{U} \setminus \partial E$},
    \end{equation}
where $\gamma_{1,s}$ is again defined by \eqref{eq: constant fractional laplacian}. Next, we aim to move the limit under the integral and use the two identities above to conclude that
    \begin{equation} \label{eq:intro-A}
        \lim_{\eps \to 0^+} \G_{s,\eps}(u_\eps, \mathcal{U}) = \int_{\mathcal{U}} \bigg( (-\Delta)^s \chi_E - \frac{\gamma_{1,s}}{s} \frac{\chi_{E}}{\abs{\d_{\partial E}}^{2s}} \bigg)^2 \, dx.
    \end{equation}
This is proved if we exclude the possibility of energy concentration on $\partial E$. To this end, we rely on the expansion of the fractional Laplacian in Fermi coordinates, in the spirit of~\cite{CFI24}*{Theorem~4.1}, following the earlier works~\cites{CLW17,ChanWei17}. More precisely, if we identify any $x \in \mathcal{U}$ with a pair $(y,z)$, where $y \in \partial E$ is the projection of $x$ onto $\partial E$ and $z = \d_{\partial E}(x)$, then for any $s \in \left[\sfrac{1}{2},1\right)$ the following expansion holds
    \begin{equation}
        \label{eq:fL-Fermi}
        (-\Delta)^s u_{\eps}(x) + \frac{W'(u_\eps(x))}{\eps^{2s}} = H_{\partial E}(y) L_s[w_{\eps}'](z) + \mathcal{R}_{\eps}(y,z),
    \end{equation}
where $L_{s}$ is a suitable operator evaluated at $w'_{\eps}(z):= \eps^{-1} w'\left(\sfrac{z}{\e}\right)$ and $\mathcal{R}_{\eps}(y,z)$ is an error term. Given $s \in \left[\sfrac{3}{4}, 1\right)$, for every $\delta > 0$, the $L^2$-norm of $L_s[w_\eps']$ on $(-\delta,\delta)$ blows up as $\eps \to 0$, thereby allowing for energy concentration. In contrast, for $s \in \left[ \sfrac{1}{2}, \sfrac{3}{4}\right)$, we will show in \cref{p: constant willmore} that
    \begin{equation}
        \label{eq:small-o-delta}
        \limsup_{\varepsilon \to 0^+} \int_{-\delta}^{\delta} L_s[w_\eps'](z)^2 \, dz = o(1),
        \quad \text{as $\delta \to 0^+$}.
    \end{equation}
In addition, for regular sets, the right-hand side of~\eqref{eq:intro-A} is finite and strictly positive precisely when $s < \sfrac{3}{4}$ (see Proposition~\ref{prop:N-delta}). This quantity can be interpreted as a $L^2$-analogue (hence ``thick'', as opposed to pointwise) of the fractional mean curvature of $\partial E$. In particular, it vanishes if $E$ is a half-plane (see \cref{cor:N vanishes-2}). 

Even though identity~\eqref{eq:intro-A} confirms that the scaling in the definition of $\G_{s,\varepsilon}$ is the appropriate one, this is not enough to conclude the proof of Theorem~\ref{t:main}. To this end, the key remark is that a version of the identity~\eqref{eq:intro-A} still hold whenever $u_\eps$ is replaced by 
    \begin{equation}
        u_{f,\varepsilon}(x) := w \bigg( \frac{f(x)}{\varepsilon}\bigg).
    \end{equation}
Here, $f \colon \R^d \to \R$ is a well-behaved defining function for $\partial E$, that is, $E = \{ f > 0 \}$, and the gradient of $f$ does not vanish on $\partial E = \{ f = 0 \}$. In this case, the only difference is that $f$ appears on the right-hand side of~\eqref{eq:intro-A} instead of the signed distance function $\d_{\partial E}$. Finally, by optimizing over all possible choices of $f$, we obtain the conclusion of Theorem~\ref{t:main}.

The same heuristic argument applies also to the case $s \in \left( 0, \sfrac{1}{2}\right)$. As a matter of fact, the expansion of the fractional Laplacian considerably simplifies, since the contribution due to the operator $L_s[w_\e']$ is lower order with respect to the error term $\mathcal{R}_{\e}$.

\subsection{Further remarks}

We conclude this section with some comments. 

a) The problem of minimizing the Willmore functional (i.e., the right-hand side of~\eqref{eq:liminf-intro}) within prescribed classes of surfaces in $\R^3$ has attracted considerable attention over time~\cites{S93, K96, BK03, S12, DMR14,DLW17}. The validity of~\eqref{eq:liminf-intro} would have provided a PDE-based route to establishing the existence of such minimizers. As a consequence of our main result this approach is not suitable when $s \in \left[\sfrac{1}{2},\sfrac{3}{4}\right)$.

b) While Theorem~\ref{t:main} rules out the validity of~\eqref{eq:liminf-intro} for $s \in \left[\sfrac{1}{2},\sfrac{3}{4}\right)$, one may still ask whether the inequality holds for sequences of critical points of $\F_{s,\eps}$ with uniformly bounded energy. This leads to the following question. Let $E \subset \R^d$ be a set of finite perimeter, and let $\{ u_\eps \}$ be a family of functions such that $u_\eps \to \chi_E$ in $L^1_{loc}(\R^d)$, and  
\begin{equation} \label{eq:Question-C}
    (-\Delta)^s u_\eps + \frac{W'(u_\eps)}{\eps^{2s}} = 0, \quad \text{with} \quad \limsup_{\eps \to 0^+} \F_{s,\eps}(u_\eps, \Omega) < +\infty.
\end{equation}
Is it true that the reduced boundary $\partial^* E \cap \Omega$ has zero mean curvature in a suitable weak sense? In the local case $s=1$, this property holds in any dimension, as shown in the seminal work of Hutchinson and Tonegawa~\cite{HT00}. 

On the other side, when $s \in \left(0,\sfrac{1}{2}\right)$ the situation is different. In~\cite{MSW19}, it is proved that if $\{u_{\eps}\}$ is a family of functions satisfying~\eqref{eq:Question-C}, then $u_\eps \to \chi_{E_{*}}$, where $E_{*}$ is a set such that $\partial E_{*} \cap \Omega$ is a $2s$-fractional minimal surface in $\Omega$. Roughly speaking, $\partial E_{*}$ has vanishing $2s$-fractional mean curvature inside $\Omega$. In addition, they obtained more general compactness results for families of functions with equibounded energy and a uniform Sobolev bound on their first variations.

c) The conclusion of Theorem~\ref{t:main} resembles the behavior of another family of functionals that is closely related to $\G_{1,\eps}$. In~\cite{DG91}, De Giorgi introduced the following functionals
    \begin{equation}
        \label{eq.DG-original}
        \bar{\mathcal{G}}_\eps(u, \Omega):=\int_{\Omega} \bigg(\eps\Delta u-\frac{W'(u)}{\eps} \bigg)^2 \, d\mu_\eps, \qquad d\mu_\eps:=\bigg(\eps\frac{|\nabla u|^2}{2} + \frac{W(u)} {\eps} \bigg) \, dx.
    \end{equation}
Moreover, he conjectured that the Gamma-limit of $\F_{1,\eps} + \bar{\mathcal{G}}_\eps$ coincides (up to positive constants) with the sum of the perimeter and the Willmore functional, at least for smooth sets. The difference between $\bar{\G}_{\eps}$ and $\G_{1,\eps}$ lies in the fact that the measure $\mu_\eps$ is replaced by $\sfrac{1}{\eps} $ times the Lebesgue measure. In fact, for the same sequence considered in~\eqref{eq:limsup-local-1}, the density of $\mu_\eps$ is proportional to $\sfrac{1}{\eps}$ near $\partial E$, whereas both its contribution and that of its first variation become negligible away from the interface. Therefore, the two functionals were expected to have the same asymptotic behavior. However, as shown in~\cite{BFP23}, inequality~\eqref{eq:liminf-intro} with $\bar{\G}_{\eps}$ in place of $\G_{1,\eps}$ cannot hold.

\subsection{Plan of the paper} In~\cref{s: tools}, we recall the notion of fractional perimeter and review some properties of the one-dimensional optimal profile, as well as of the signed distance function from a smooth set. We also establish here the main lemmas needed for the proof of~\cref{t:main}. In~\cref{s:expansion of fractional laplacian}, we adapt some ideas from~\cite{CFI24} to derive the expansion, in Fermi coordinates, of the $s$-fractional Laplacian around the boundary of a smooth set when $s < \sfrac{3}{4}$. Finally, in~\cref{s:proof of main theorem}, we complete the proof of~\cref{t:main}. 

\section{Tools} \label{s: tools}

Given $s \in (0,1)$ and $u \colon \R^d \to \R$, the $s$-fractional Laplacian of the function $u$ is defined by
\begin{equation}
    (-\Delta)^s u(x) := \gamma_{d,s} \lim_{\nu \to 0^+} \int_{B_\nu(x)^c} \frac{u(x) -u(y)}{\abs{x-y}^{d+2s}}\, dy  \label{eq: fractional laplacian},
\end{equation}
where 
\begin{equation}
    \gamma_{d,s} := s 2^{2s} \pi^{-\frac{d}{2}} \frac{\Gamma\left( \frac{d+2s}{2} \right)}{\Gamma(1-s)}. \label{eq: constant fractional laplacian} 
\end{equation}
For every $u \in C^2(\Omega) \cap L^\infty(\R^d)$ and every $x \in \Omega$, the limit on the right-hand side of~\eqref{eq: fractional laplacian} exists. We refer the reader to~\cites{Ga19, DPV12} for further details. 

Throughout this note, we will always consider a double-well potential $W$ satisfying the following structural assumptions:  
    \begin{enumerate}[label=($W1$),ref=$W1$]
        \item\label{h: zero of potential}
        $W \colon \R \to [0, +\infty)$, $W$ is even and $ \{W = 0 \} = \{\pm 1\}$; 
        \end{enumerate}
        \begin{enumerate}[label=($W2$),ref=$W2$]
        \item\label{h: potential is smooth}
        $W \in C^{3}(\R)$; 
        \end{enumerate}
        \begin{enumerate}[label=($W3$),ref=$W3$]
        \item \label{h: W'' positive} 
        $W''(\pm 1) = \lambda > 0$; 
    \end{enumerate} 
For any open set $E \subset \R^d$, we set $\Sigma : = \partial E$. We denoty by $\chi_E$ the function defined by $\chi_E(x) = 1$ if $x \in E$, and $\chi_E(x)=-1$ if $x \notin E$.

\subsection{The fractional perimeter} \label{ss: fractional perimeter}

For every $\sigma \in (0,1)$, the $\sigma$-fractional perimeter of a measurable set $E \subset \R^d$ inside the open set $\Omega$ is defined as
    \begin{equation}
        \label{eq:def-frac-per}
            \Per_{\sigma}(E, \Omega)
            := 2
            \iint_{E \times E^c \setminus (\Omega^c \times \Omega^c)} \frac{dx dy}{\abs{x-y}^{d+\sigma}}.
    \end{equation}
The first use of such functionals as an alternative to the classical perimeter dates back to Vi\-sin\-tin~\cite{Vis90} to model solid-liquid systems with irregular interfaces. They were later popularised after the foundational work on nonlocal minimal surfaces by Caffarelli, Roquejoffre, and Savin~\cite{CRS}.

For every $\lambda > 0$, it holds that $\Per_{\sigma}(\lambda E, \R^d) = \lambda^{d-\sigma} \Per_{\sigma}(E,\R^d)$. Hence, as $\sigma \to 1$, the scaling behavior of the $\sigma$-fractional perimeter becomes closer to that of the classical perimeter. In particular, for any set $E$ of finite perimeter in $\R^d$ it holds that (see \cite{Davila-BBM})
    \begin{equation}
        \label{eq:pointwise-conv}
        \lim_{\sigma \to 1^-} (1-\sigma) \Per_{\sigma}(E, \R^d) = \frac{2\omega_{d-2}}{d-1} \Per(E,\R^d), \qquad \omega_{d-2}:=\mathcal{H}^{d-2}(\mathbb{S}^{d-2}).
    \end{equation}
Moreover, this approximation also holds in the sense of Gamma-convergence~\cites{Ponce-calcvar, Gammaconv}. Another reason explaining the name fractional perimeter is that it satisfies a generalized coarea formula (see, for example,~\cite{Lomb18}). We also mention that, recently, some notions of fractional area in codimension higher than one, satisfying an analogue of~\eqref{eq:pointwise-conv}, have been proposed~\cites{FM1, FM2, CP25}, and in some cases the approximation is also valid in the sense of Gamma-convergence~\cite{CFP24}.

As for the classical perimeter, sets of finite $\sigma$-perimeter in a Lipschitz domain can be approximated by smooth sets. The following is part of a characterization result from~\cite{Lomb18}*{Theorem 1.3}.


\begin{lemma} \label{l: approximation of fractional perimeter perimeter}
Fix $\sigma \in \left( 0, 1\right)$. Let $\Omega \subset \R^d$ be a bounded Lipschitz domain and let $E \subset \R^d$ be a set with finite $\sigma$-perimeter in $\Omega$. There exists a sequence $\{E_n\}$ of smooth open sets converging to $E$ in $L^1_{loc}(\R^d)$ such that 
\begin{equation}
    \lim_{n \to +\infty} \Per_{\sigma}(E_n, \Omega) = \Per_{\sigma} (E, \Omega). 
\end{equation}
Moreover, if $E$ is bounded, then the sets $E_n$ can also be chosen to be bounded.
\end{lemma}

\subsection{The one-dimensional optimal profile} \label{ss: optimal profile }

We recall that $w \colon \R \to (-1,1)$ denotes the unique increasing solution to the fractional Allen--Cahn equation 
\begin{equation} \label{eq: fractional AC}
\begin{cases}
    (-\partial_{zz} )^s w + W'(w) = 0, 
    \\[0.5ex] w(0)=0, 
    \\[0.5ex] \displaystyle \lim_{z \to \pm \infty} w(z) = \pm 1, 
\end{cases}
\end{equation}
where $(-\partial_{zz})^s$ is the $s$-fractional Laplacian in dimension one. We recall some properties of the one-dimensional optimal profile relevant to our analysis.

\begin{proposition} \label{t:optimal profile}
Let $s \in (0,1)$ and let $W$ be a double-well potential. 
Then, there exists a unique strictly increasing solution $w \colon \R \to (-1,1)$ to the problem \eqref{eq: fractional AC}. Moreover, $w$ is odd, $w \in  C^2(\R)$ and there exists $C(s,W)>0$ such that  
\begin{equation}
    \abs{z \cdot w''(z)} + \abs{ w'(z)} \leq \frac{C}{1+ \abs{z}^{1+2s}}, \qquad \forall z \in \R. \label{eq: decay w'}
\end{equation}
\end{proposition}

We recall that $w$ is the unique (up to translation) minimizer of $\mathcal{F}_{s,1}(\cdot, \R)$ with respect to compact perturbations in $\R$ (see, for instance,~\cite{PSV13}*{Theorem 2}). Since $W$ is even, it is readily checked that $w$ is odd. The regularity of $w$ is established by iterating the apriori estimates in~\cite{S07}*{Proposition 2.8-2.9}. The decay of higher derivatives of $w$ has been studied in~\cite{CFI24}*{Theorem 1.2}. The estimate~\eqref{eq: decay w'} is optimal, see \cite{CSbis14}*{Theorem 2.7}. The following computation will be useful.

\begin{lemma} \label{l: limit W'(u_e)}
Let $s \in (0,1)$ and set $w_\e(z) = w\left( \sfrac{z}{\e}\right)$. Then, for any $\alpha>0$ we have
\begin{equation} \label{eq: limit of W'(u_e)}
    \lim_{\e \to 0^+}\frac{W'(w_\e(z))}{\eps^{2s}} = - \frac{\gamma_{1,s}}{s} \frac{z}{\abs{z}^{1+2s}} \qquad \text{uniformly on $\{\abs{z} \geq \alpha\}$}.    
\end{equation}
\end{lemma}

\begin{proof}
We let
$$\tilde{w}(z) : = \frac{w(z) +1}{2}\qquad{\mbox{and}} \qquad \tilde{W}(z) := \frac{1}{4 \gamma_{1,s}} W(2z-1). $$
Then, $\tilde{W}$ has wells at $z =0, z=1$ and $\tilde{w}$ is the unique increasing solution to  
\begin{equation} \label{eq: modified profile}
    \begin{cases}
        \gamma_{1,s}^{-1} (-\Delta)^s \tilde{w} + \tilde{W}'(\tilde{w}) = 0, 
        \\ \displaystyle \tilde{w}(0) = \frac{1}{2}, 
        \\ \displaystyle\lim_{z\to \pm \infty } \tilde{w}(z) = \mathds{1}_{(0, +\infty)} (\pm \infty).    
    \end{cases}
\end{equation}
Therefore, by \cite{DPV15}*{Equation 1.6} we infer\footnote{The choice of adding the constant $\gamma_{1,s}$ in the definition of $\tilde{W}$ and in \eqref{eq: modified profile} is due to the fact that $\gamma_{1,s}$ is not encoded in the definition the fractional Laplacian in \cite{DPV15}, as opposed to ours in \eqref{eq: fractional laplacian}.}
\begin{equation}
    w(z) = 2 \cdot \mathds{1}_{(0, +\infty)} (z) - 1 -\frac{\gamma_{1,s}}{s W''(1)} \frac{z}{\abs{z}^{1+2s}} + O(\abs{z}^{-1-2s}) \qquad \abs{z} \to \infty. 
\end{equation}
Since $w_\e(z) = w \left( \sfrac{z}{\e}\right)$ and  $2 \cdot \mathds{1}_{(0, +\infty)} (z) - 1 = \sgn(z)$, for any $z >0 $ we compute 
\[\begin{split}
W'(w_\e(z)) & =  W'\left(1-\frac{\gamma_{1,s} }{s W''(1)}\frac{\eps^{2s} z}{\abs{z}^{1+2s}} + O\left(        \frac{\eps^{1+2s}}{\abs{z}^{1+2s}} \right) \right)
\\ & = W'(1)+W''(1)\left( -\frac{\gamma_{1,s} }{sW''(1)}\frac{\eps^{2s} z}{\abs{z}^{1+2s}} + O\left( \frac{\eps^{1+2s}}{\abs{z}^{1+2s}} \right) \right) +O\left( \frac{\eps^{4s}}{\abs{z}^{4s}} \right).
\end{split}\]
Recalling that $W'(1) = 0$, then the result follows. The calculations are analogous on $z<0$. 
\end{proof}

\subsection{Distance function} \label{ss: distance function}

Throughout this note, we adopt the following notation. Given an open set $E$, we let $\Sigma = \partial E$ and~$\dist_{\Sigma}$ be the signed distance from $\Sigma$, that is
\begin{equation}
    \dist_{\Sigma} (x) := \begin{cases}
        \inf \{ \abs{x-y} \colon y \in \Sigma \} & \text{ if } x \in E, 
        \\[0.5ex] - \inf \{ \abs{x-y} \colon y \in \Sigma \} & \text{ if } x \in E^c. 
    \end{cases}
\end{equation}

It is well known that the distance function is $1$-Lipschitz continuous and $\abs{ \nabla \dist_{\Sigma}(y)} =1$ at any point of differentiability. We denote by $\Sigma_{\ell} = \{ x \in \R^d : \abs{\mathrm{d}_{\Sigma}(x)} \le \ell \}$, for every $\ell > 0$. We recall some well-known properties of the signed distance function from a hypersurface (we refer to~\cite{Ambrosio},~\cite{GT01}*{Lemma 14.16}, and~\cite{Modica}*{Lemma 3}).

\begin{lemma} \label{l: regularity of distance function}
Let $E \subset \R^d$ be a bounded open set of class $C^k$ for some $k \geq 2$. Then, there exists $\delta>0$ depending only on $\Sigma$ with the following properties.
\begin{itemize}
    \item [(i)] For any $x \in \Sigma_{\delta}$ there exists a unique point $\pi_{\Sigma}(x) \in \Sigma$ of minimal distance between $x$ and $\Sigma$. Moreover, the map $\pi_{\Sigma} \colon \Sigma_\delta \to \Sigma$ is of class $C^{k-1}$.
    \item [(ii)] For any $x \in \Sigma_\delta$ it holds $ \nabla \dist_{\Sigma}(x) = N(\pi_{\Sigma}(x)) $, where $N$ is the inner unit normal to $\Sigma$. In particular, $\dist_{\Sigma} \in C^{k}(\Sigma_\delta)$. 
\end{itemize}
\end{lemma}

We introduce a globally smooth version of the signed distance function, which will be crucial in the sequel. For every $\Sigma$ as above, we fix $\delta_0=\delta_0(\Sigma) > 0$ such that $\d_\Sigma$ is $C^2(\Sigma_{5 \delta_0})$. 

\begin{definition} \label{d:regular distance} 
Let $E$ be a bounded open set of class $C^2$, and let $\Sigma = \partial E$. For every function $\eta \in L^\infty(\R^d)$ and $\delta \in (0,\delta_0)$, we define 
\begin{equation} \label{eq: definition beta}
    \beta_\Sigma^{\eta,\delta} : = \begin{cases}
        \d_\Sigma & x \in \Sigma_\delta, 
        \\[0.5ex] \eta & x \in (\Sigma_\delta)^c. 
    \end{cases}
\end{equation}
Moreover, we define the set 
\begin{equation}
    \mathcal{K}_\delta := \Big\{ \eta \in  L^\infty (\Sigma_\delta^c) : \sgn \eta = \sgn \d_\Sigma, \, \inf_{x \in \Sigma_\delta^c} \abs{\eta(x)} >0, \, \beta_\Sigma^{\eta,\delta} \in C^2(\R^d) \Big\}. 
\end{equation}
\end{definition}

As explained in the introduction and in view of \Cref{l: limit W'(u_e)}, the supposed limit of $\mathcal{G}_{s,\eps}$ under the ``standard'' recovery sequence is the following.

\begin{definition}
Let $s \in (0,1)$ and consider a bounded open set $\Omega\subset \R^d$. For any bounded open set $E \subset \R^d$, we define
    \begin{equation}  \label{eq:frac-Bel-Pao-limit}
        \cN_{s}(\chi_E, \Omega) : =  \int_{\Omega} \Bigg( (-\Delta)^s \chi_E -\frac{\gamma_{1,s}}{s} \frac{\chi_E}{\abs{\d_\Sigma}^{2s}} \Bigg)^2 \, dx. 
    \end{equation}    
\end{definition}
In the following result we show that $\cN_{s}(\chi_E, \Omega)$ is finite precisely when $s < \sfrac{3}{4}$ and $E$ is smooth. 

\begin{proposition} \label{prop:N-delta}
Let $s \in \left(0, \sfrac{3}{4}\right)$ and $E \subset \R^d$ be a bounded open set of class $C^2$. Then, the functional $\cN_s(\chi_E, \Omega)$ defined by \eqref{eq:frac-Bel-Pao-limit} is finite. 
\end{proposition}

\begin{proof} 
Let $\bar\delta>0$ be less than the geometric parameter given by \cref{l: regularity of distance function}. For all small enough $\delta\in(0,\bar\delta)$ and $x \in \R^d \setminus \Sigma_\delta$, we have 
\begin{equation}
    \abs{(-\Delta)^s \chi_E (x)} = \gamma_{d,s} \abs{ \int_{B_{\delta}(x)^c} \frac{\chi_E(x) - \chi_E(y)}{\abs{x-y}^{d+2s}} \, dy } \le \frac{\gamma_{d,s}  \omega_{d-1}}{s} \delta^{-2s}.  
\end{equation}
Hence, 
\begin{equation}
    \sup_{x \in \R^d \setminus \Sigma_\delta} \abs{ (-\Delta)^s \chi_E(x) - \frac{\gamma_{1,s}}{s } \frac{\chi_E(x)}{\abs{\d_\Sigma(x)}^{2s}}} \le \frac{1}{s} \big( \gamma_{d,s} \omega_{d-1} + \gamma_{1,s} \big) \delta^{-2s}. 
\end{equation}
To estimate the integrand in $\Sigma_\delta$, we need to gain one power of $\d_\Sigma$ by exploiting cancellations. We consider $x \in \Sigma_\delta \cap E$ and, by \cref{l: regularity of distance function}, we find a unique point $\bar{x} \in \Sigma$ of minimal distance $\d_\Sigma(x) > 0$. Without loss of generality, we may assume that $\bar{x} = 0$, $x=(0, \d_\Sigma(x))$. Letting $y = (y', y_d) \in \R^{d-1} \times \R$, by \cite{CFI24}*{Corollary 7.2} we compute 
\begin{align} 
    \gamma_{d,s}\int_{\{y_d <0 \}} \frac{2}{\abs{x-y}^{d+2s}} \, dy & = \gamma_{d,s} \int_{-\infty}^0 \int_{\R^{d-1}} \frac{2}{( \abs{y'}^2 + \abs{y_d - \d_\Sigma(x)}^2 )^{\frac{d+2s}{2}} } \, dy' dy_d
    \\ & = 2 \gamma_{1,s} \int_{-\infty}^0 \frac{1}{\abs{\d_{\Sigma}(x) - y_d}^{2s +1}} \, d y_d 
    \\ & = 2 \gamma_{1,s} \int_{\d_\Sigma(x)}^{+\infty} \frac{1}{t^{2s+1}} \, dt  = \frac{\gamma_{1,s}}{s} \frac{1}{\d_\Sigma(x)^{2s}}.  \label{eq: reduction of the integral} 
\end{align}
Since $\Sigma$ is compact and it has $C^{2}$ regularity, we find geometric constants $c_1, c_2>0$ (uniform with respect to $x \in \Sigma_\delta$) such that
    \[
        (E^c \Delta \{ y_d <0 \}) \cap \{ \abs{y} < c_1 \} \subset \{ \abs{y_d}\leq c_2 \abs{y'}^2 \} \cap \{ \abs{y} < c_1\}.
    \]
Therefore, by \eqref{eq: reduction of the integral} we compute 
\begin{align}
\frac{1}{2\gamma_{d,s}}\abs{ (-\Delta)^s \chi_E - \frac{\gamma_{1,s}}{s} \frac{\chi_E}{\abs{\d_\Sigma}^{2s}} } = & \abs{ \int_{E^c}  \frac{dy}{\abs{x-y}^{d+2s}} -\int_{\{y_d < 0\}} \frac{dy}{\abs{x-y}^{d+2s}} }
\\[1ex] \le & \int_{(E^c \Delta \{ y_d <0 \}) \cap \{ \abs{y} < c_1 \}} \frac{dy}{\abs{x-y}^{d+2s}} + \underbrace{\int_{\{ \abs{y} > c_1 \}} \frac{dy}{\abs{x-y}^{d+2s}}}_{=:I_3} 
\\[1ex] \le & \int_{B_{c_1}^{d-1}} dy' \int_{ \{ \abs{y_d} \le c_2 \abs{y'}^2 \} } \frac{d y_d}{(\abs{y'}^2 + (\d_\Sigma(x)-y_d)^2)^{\frac{d+2s}{2}}} + I_3 
\\[1ex] = & \frac{1}{\d_{\Sigma}(x)^{2s}} \int_{ B^{d-1}_{c_1 / \d_\Sigma(x)}} dz' \int_{A_x} \frac{d z_d}{(\abs{z'}^2 + (1-z_d)^2)^{\frac{d+2s}{2}}} + I_3 
\\[1ex] = & \underbrace{\frac{1}{\d_{\Sigma}(x)^{2s}} \int_{B^{d-1}_{c_1 / \d_\Sigma(x)}} dz' \int_{A_x^1} \frac{d z_d}{(\abs{z'}^2 + (1-z_d)^2)^{\frac{d+2s}{2}}}}_{=:I_1} 
\\[0.5ex] & +  \underbrace{\frac{1}{\d_{\Sigma}(x)^{2s}} \int_{B^{d-1}_{c_1 / \d_\Sigma(x)}} dz' \int_{A_x^2} \frac{d z_d}{(\abs{z'}^2 + (1-z_d)^2)^{\frac{d+2s}{2}}}}_{=:I_2} 
\\[0.5ex] & + I_3,
\end{align} 
where $A_x:=\left\{\abs{z_d} \le c_2 \d_{\Sigma}(x) \abs{z'}^2 \right\}$, $A_x^1:= A_x \cap \left\{ \abs{z_d - 1} \le \sfrac{1}{2} \right\}$, and $A_x^2:= A_x \cap \left\{ \abs{z_d - 1} > \sfrac{1}{2} \right\}$. To estimate $I_1$, since $A_x^1 \subset \left( \sfrac{1}{2}, \sfrac{3}{2} \right)$, we have 
\begin{align} \label{eq:N-s I-1}
    I_1 \leq \abs{\d_\Sigma}^{-2s} \int_{\frac{1}{2}}^{\frac{3}{2}} \int_{ \{ \abs{z'} \geq (2c_2 \d_\Sigma)^{-\frac{1}{2}} \}} \frac{1}{\abs{z'}^{d+2s} } \, d z' d z_d \lesssim \abs{\d_\Sigma}^{\frac{1-2s}{2}}. 
\end{align}
To estimate $I_2$, we have 
\begin{align}  
    I_2 & \lesssim \abs{\d_\Sigma}^{-2s} \int_{\abs{z'} \leq c_1\d_\Sigma^{-1} } \int_0^{c_2 \d_\Sigma \abs{z'}^2} \frac{1}{(1+ \abs{z'}^2)^{\frac{d+2s}{2}}} \, d z_d  dz' \lesssim \abs{\d_\Sigma}^{1-2s} \int_{\abs{z'} \leq c_1 \d_\Sigma^{-1}} \frac{ \abs{z'}^2}{(1+\abs{z'}^2)^{\frac{d+2s}{2}}} \, dz'.   
\end{align}
Therefore, by a direct computation, we find that
\begin{equation} \label{eq:N-s I-2}
    I_2 \lesssim \begin{cases}
        \abs{\d_\Sigma}^{1-2s} & s \in \left( \sfrac{1}{2}, 1 \right), 
        \\[0.5ex] \abs{\log(\d_\Sigma)} & s= \sfrac{1}{2}, 
        \\[0.5ex] 1 & s \in \left( 0, \sfrac{1}{2} \right). 
    \end{cases}
\end{equation}
To estimate $I_3$, since $c_1, c_2$ are geometric constants depending only on $\overline{\delta}$, then $\delta$ can be chosen so that $\delta \leq \sfrac{c_1}{2}$. As a result, one sees that 
\begin{equation} \label{eq:N-s I-3}
    I_3 \lesssim \int_{B_{\sfrac{c_1}{2}}} \frac{1}{\abs{y}^{d+2s}} \, dy \lesssim 1. 
\end{equation}
To summarize, by \eqref{eq:N-s I-1}, \eqref{eq:N-s I-2}, \eqref{eq:N-s I-3}, for any $x \in \Sigma_\delta \cap E$ we estimate 
\begin{equation}
    \abs{ (-\Delta)^s \chi_E - \frac{\gamma_{1,s}}{s} \frac{\chi_E}{\abs{\d_\Sigma}^{2s}} } \lesssim \begin{cases}
        \abs{\d_\Sigma}^{1-2s} & s \in \left( \sfrac{1}{2}, 1 \right), 
        \\[0.5ex] \abs{\log(\d_\Sigma)} & s= \sfrac{1}{2}, 
        \\[0.5ex] 1 & s \in \left( 0, \sfrac{1}{2} \right),  
    \end{cases}
\end{equation}
which is in $L^2(\Sigma_\delta \cap E)$ precisely for $s < \sfrac{3}{4}$. The computation in $\Sigma_\delta \cap E^c$ is analogous.  
\end{proof}

By the proof of \cref{prop:N-delta} we obtain the following result, which holds for $s \in (0,1)$.  

\begin{corollary} \label{cor:N vanishes-2}
Let $\Omega$ be any open set and $s \in (0,1)$. 
\begin{enumerate}
    \item For any half-space $H \subset \R^d$, it holds $\mathcal{N}_{s}(\chi_H, \Omega) = 0$.
    \item Viceversa, assume that $E$ is an 
    open set such that $\mathcal{N}_s(\chi_E, \Omega)= 0 $, and there exists a half-space $H \supset E$ such that $\partial E\cap \partial H \cap \Omega\neq \varnothing$.\footnote{i.e. global inclusion and touching inside $\Omega$. For instance, this is satisfied whenever $E$ is convex or $\partial E \subset \Omega$ and $E$ is of class $C^1$. } 
    Then, $E$ is a half-space. 
\end{enumerate}
\end{corollary}

\begin{proof}
\begin{enumerate}
\item
By the computation in \eqref{eq: reduction of the integral}, we have
    \begin{equation}\label{eq:Ns-H}
    (-\Delta)^s \chi_H(x) = \frac{\gamma_{1,s}}{s} \frac{\chi_H(x)}{\abs{\d_{\partial H}(x)}^{2s}} \quad \text{ for all } x\in \R^d \setminus \partial H.
    \end{equation}
In particular, $\cN_s(\chi_H, \Omega) = 0$.
\item Since $\cN_s(\chi_E,\Omega)=0$, then we infer 
    \begin{equation}\label{eq:Ns-E}
    (-\Delta)^s \chi_E(x) = \frac{\gamma_{1,s}}{s} \frac{\chi_E(x)}{\abs{\d_{\Sigma}(x)}^{2s}} \quad \text{ for all } x\in \Omega \setminus \Sigma.
    \end{equation}
    Without loss of generality, we may assume that $H = \{ x_d >0 \}$ and $0 \in \partial E \cap \partial H \cap \Omega$. Hence, we find a point $x_0 = -\ell e_d$, for some $\ell >0$ small enough, such that $x_0 \in \Omega\setminus \overline{H} \subset \Omega\setminus \overline{E}$ and  $\d_{\partial H}(x_0)=\d_{\Sigma}(x_0)$. Thus, we have
    \[
    0 =
    (-\Delta)^s \chi_E(x_0)-(-\Delta)^s \chi_H(x_0)
    = \gamma_{d,s} \int_{H \setminus E} \frac{2}{\abs{x_0 - y}^{d+2s}} \, dy,
    \]
    forcing $E=H$, as desired.
\end{enumerate}
\end{proof}

\section{On the expansion of the fractional Laplacian around the boundary} \label{s:expansion of fractional laplacian}

We describe the expansion of the fractional Laplacian of power $s \in \left( 0, \sfrac{1}{2}\right]$ in Fermi coordinates for the function defined by \eqref{eq: recovery sequence}. The latter complements the result in \cite{CFI24}*{Theorem 4.1} where the case $s \in \left( \sfrac{1}{2},1\right)$ has been considered. For the sake of clarity, below we give a complete statement. The case $s = \sfrac{1}{2}$ is a straightforward modification of \cite{CFI24}*{Theorem 4.1}, whereas the expansion for $s \in \left(0, \sfrac{1}{2}\right)$ is easier. A possible heuristic explanation is the following. For $s < \sfrac{1}{2}$, the second term is proportional to $\e^{1-2s}$, thus being of lower order with respect to the remainder term. Hence, in the computations below, the integrals at the origin are easier to estimate for $s$ within this regime. The following statement, in the case $s < \sfrac{3}{4}$, will suffice our purposes.

\begin{theorem}\label{t:fractional laplacian}
Let $s \in(0, 1)$ and let $W$ be a double-well potential. Let $w: \R \to (-1,1)$ be the one-dimensional optimal profile and for any $\e>0$ let us set $w_\e(z) = w \left( \sfrac{z}{\e}\right)$. Let $E$ be a bounded open set of class $C^3$. Denote by $\Sigma = \partial E$ and let $\delta_0 >0$ be as in \cref{d:regular distance}. For $\delta < \delta_0$ and $\eta \in \mathcal{K}_\delta$, define 
\begin{equation} \label{eq: recovery sequence}
     u_\e^{\eta, \delta} (x) : = w_\e (\beta_\Sigma^{\eta, \delta}(x)),   
\end{equation}
where $\beta_\Sigma^{\eta, \delta}$ is as in \eqref{eq: definition beta}. There exists $\Lambda_0, C \geq 1$ depending on $\eta, \delta$ such that for any $\Lambda \geq \Lambda_0$ and for any $\e \in (0,1)$ there exists $\mathcal{R}_{\e,\Lambda}^{\eta, \delta}: \Sigma_{\sfrac{\delta}{10 \Lambda}} \to \R$ with the following properties. 
\begin{itemize}
\item[$(i)$] If $s < \sfrac{1}{2}$, it holds 
\begin{align}
        (-\Delta)^s u_\e^{\eta, \delta}(x_0) & = (-\partial_{zz})^s w_\e(z_0) + \mathcal{R}_{\e, \Lambda}^{\eta, \delta}(x_0) \label{eq:expansion s< 1/2} 
\end{align} 
for any $x_0 \in \Sigma_{\sfrac{\delta}{10 \Lambda}}$, where we set $z_0 = \dist_{\Sigma}(x_0)$. The reminder satisfies 
\begin{equation} \label{eq: bound reminder s<1/2}
    \norm{\mathcal{R}_{\e,\Lambda}^{\eta, \delta}}_{L^\infty(\Sigma_{\sfrac{\delta}{10 \Lambda}})} \leq C \Lambda^{2s}. 
\end{equation}
    \item[$(ii)$] If $s = \sfrac{1}{2}$, it holds
\begin{align}
        (-\Delta)^\frac{1}{2} u_\e^{\eta, \delta}(x_0) & =  (-\partial_{zz})^\frac{1}{2} w_\e(z_0) - \frac{\gamma_{1,\sfrac{1}{2}}}{2 } H_{\Sigma}(x_0') \int_{-\sfrac{\delta}{\Lambda}}^{\sfrac{\delta}{\Lambda}} w_\e'(z_0+\bar{z}) \log(\abs{\bar{z}}) \, d\bar{z} + \mathcal{R}_{\e, \Lambda}^{\eta, \delta}(x_0) \label{eq:expansion s=1/2} 
\end{align} 
for any $x_0 \in \Sigma_{\sfrac{\delta}{10 \Lambda}}$, where we set $z_0 = \dist_{\Sigma}(x_0)$ and $x_0' = \pi_{\Sigma} (x_0)$. The reminder satisfies 
\begin{equation} \label{eq: bound reminder s=1/2}
    \norm{\mathcal{R}^{\eta, \delta}_{\e,\Lambda}}_{L^\infty(\Sigma_{\sfrac{\delta}{10 \Lambda}})} \leq C \Lambda \log (\Lambda).
\end{equation}
\item[(iii)]
If $s\in(\sfrac{1}{2},1)$, it holds
\begin{align}
        (-\Delta)^s u_\e(x_0) & = (-\partial_{zz})^s w_\e(z_0) + \frac{\gamma_{1,s}}{2} \frac{H_{\Sigma}(x_0')}{(2s-1)} \int_{-\sfrac{\delta}{\Lambda}}^{\sfrac{\delta}{\Lambda}} \frac{w_\e'(z_0+\bar{z})}{\abs{\bar{z}}^{2s-1}} \, d\bar{z} + \mathcal{R}_{\e, \Lambda}^{\eta,\delta}(x_0), \label{eq:expansion} 
\end{align} 
for any $x_0 \in \Sigma_{\sfrac{\delta}{10 \Lambda}}$, where we set $z_0 = \dist_{\Sigma}(x_0)$ and $x_0' = \pi_{\Sigma} (x_0)$. The reminder satisfies 
\begin{equation} \label{eq: bound reminder}
    \norm{\mathcal{R}_{\e,\Lambda}^{\eta,\delta}}_{L^\infty(\Sigma_{\sfrac{\delta}{10 \Lambda}})} \leq C \Lambda^{2s}.
\end{equation}
\end{itemize}
\end{theorem}

\begin{proof} 
The proof follows the lines of \cite{CFI24}*{Theorem 4.1}. Since $\eta, \delta$ are given, we suppress the superscript $\eta, \delta$. For the reader convenience, we recall here the main parts. The computations in the first three steps of the proof of \cite{CFI24}*{Theorem 4.1} hold for $s \in (0,1)$, since they rely only on changes of variables and Taylor expansions. To fix some notation, let $\Lambda_0 \geq 1$ be the geometric constant given by \cite{CFI24}*{Lemma 2.17}. Given $\Lambda \geq \Lambda_0$, we take $x_0 \in \Sigma_{\sfrac{\delta}{10 \Lambda}}$ and we denote by $x_0' = \pi_{\Sigma}(x_0), z_0 = \dist_\Sigma(x_0)$. Unless otherwise specified, the remainders involved in the following computations satisfy bounds depending on $\Sigma$. Then, we split the singular integral in the definition of the fractional Laplacian as follows (see \cite{CFI24}*{Equation~4.3-4.4})
\begin{align} \label{eq: outer contribution}
    (-\Delta)^{s} u_\e(x_0) = - \gamma_{d,s} \lim_{\nu \to 0^+} \Delta^{s}_\nu u_\e(x_0) + O (\Lambda^{2s}).  
\end{align}
Here, $\Delta^{s}_\nu u_\e (x_0)$ satisfies (see \cite{CFI24}*{Equation 4.21})
\begin{align}
    \lim_{\nu \to 0^+} \Delta_\nu^{s} u_\e(x_0) = & \lim_{\nu \to 0^+} \int_{\mathcal{C}_\nu} \frac{w_\e(z_0 + \bar{z} + f(z_0, \bar{y})) - w_\e(z_0) }{\rho^{d+2s}} (1- \bar{z} H_{\Sigma}(x_0') + O(\abs{z_0} \abs{\bar{y}} + \rho^2)) \, d \bar{z} d \bar{y}, \label{eq: formula Delta_nu^s}
\end{align}
where $f(z_0, \overline{y})$ satisfies (see \cite{CFI24}*{Equation 4.14})
\begin{equation} \label{eq: formula f}
    f(z_0, \overline{y}) =  \frac{1}{2} \sum_{i=1}^{d-1} k_i \abs{\bar{y}_i}^2 + O(\abs{z_0} \abs{\bar{y}}^2) + O(\abs{\bar{y}}^3), 
\end{equation}
$k_1, \dots, k_{d-1}$ are the principal curvatures of $\Sigma$ at $x_0'$ and $\mathcal{C}_\nu$ is the complement of a ball in a cylinder
\begin{equation} \label{eq: C_nu}
    \mathcal{C}_{\nu} := \Big\{ (\bar{y},\bar{z}) \in B_{\sfrac{\delta}{\Lambda}}^{d-1} \times \left(-\sfrac{\delta}{\Lambda}, \sfrac{\delta}{\Lambda}\right) \colon \rho \geq \nu \Big\}, \qquad\rho^2 = \abs{\bar{y}}^2 + \abs{\bar{z}}^2.   
\end{equation}

\textsc{\underline{Case $s < \sfrac{1}{2}$}.} For simplicity, we set $\ell = \sfrac{\delta}{\Lambda}$. Since $s \in \left( 0, \sfrac{1}{2}\right) $, by \eqref{eq: formula Delta_nu^s} we compute
\begin{align}
    \lim_{\nu \to 0^+} \Delta^s_\nu u_\e(x_0) & = \int_{-\ell}^{\ell} \int_{B_\ell^{d-1}} \frac{w_\e(z_0 + \bar{z} + f(z_0, \bar{y})) - w_\e(z_0)}{\rho^{d+2s}}\, d \bar{y} d \bar{z} + O \left( \int_{\abs{\rho} \leq \ell} \frac{1}{\rho^{d+2s-1}} \, d \rho \right) 
    \\ & =  \int_{-\ell}^\ell \int_{B_\ell^{d-1}}  \frac{w_\e(z_0+ \bar{z}) - w_\e(z_0)}{\rho^{d+2s}}\, d \bar{y} d \bar{z} 
    \\ & \qquad + \int_{-\ell}^\ell \int_{B_\ell^{d-1}} \frac{w_\e(z_0 + \bar{z} + f(z_0, \bar{y}) ) - w_\e(z_0+ \bar{z}) }{\rho^{d+2s}} \, d \bar{y} d \bar{z}  + O(\ell^{1-2s}) 
    \\ & = I_{\e, \ell} + II_{\e, \ell} + O(\ell^{1-2s}). 
\end{align} 
By an explicit computation (see e.g. \cite{CFI24}*{Corollary 7.2}) we have 
\begin{align}
    I_{\e, \ell} & = \int_{-\ell}^\ell (w_\e(z_0 + \bar{z} ) - w_\e(z_0)) \left( \frac{\gamma_{1,s}}{\gamma_{d,s}} \abs{\bar{z}}^{-1-2s} + O(\ell^{-1-2s}) \right)\, d \bar{z}
    \\ & = \frac{\gamma_{1,s}}{\gamma_{d,s}} \int_{-\ell}^\ell \frac{w_\e(z_0 + \bar{z}) - w_\e(z_0)}{\abs{\bar{z}}^{1+2s}} \, d \bar{z} + O(\ell^{-2s}) 
    \\ & = - \frac{1}{\gamma_{d,s}} (-\partial_{zz})^s w_\e(z_0) + O(\ell^{-2s}). 
\end{align}
Since \eqref{eq: formula f} implies that $\abs{f(z_0, \bar{y})} = O(\abs{\bar{y}}^2)$, by the fundamental theorem of calculus we estimate 
\begin{align}
    \abs{II_{\e, \ell} } & \lesssim \int_0^1 \int_{B_\ell^{d-1}} \int_{-\ell}^\ell w_\e'(z_0 + \bar{z} + t  f(z_0, \bar{y})) \frac{\abs{\bar{y}}^2}{\rho^{d+2s}} \, d \bar{z}  d \bar{y}  dt 
    \\ & \lesssim \int_0^1 \int_{B_\ell^{d-1}} \frac{1}{\abs{\bar{y}}^{d+2s-2}} \int_\R w_\e'(z_0 + \bar{z} + t f(z_0, \bar{y})) \, d \bar{z}  d \bar{y}  dt \lesssim O(\ell^{1-2s}). 
\end{align}
To summarize, we have shown that 
$$\lim_{\nu \to 0^+} \Delta^s_\nu u_\e(x_0) = -\frac{1}{\gamma_{d,s}} (-\partial_{zz}) w_\e(z_0) + O(\ell^{-2s}), $$
thus proving \eqref{eq:expansion s< 1/2} and \eqref{eq: bound reminder s<1/2}. 

\textsc{\underline{Case $s=1/2$}:} By \eqref{eq: formula Delta_nu^s}, we write 
\begin{align}
    \lim_{\nu \to 0^+} \Delta_\nu^{\frac{1}{2}} u_\e(x_0) = & \lim_{\nu \to 0^+} \int_{\mathcal{C}_\nu} \frac{w_\e(z_0 + \bar{z}) - w_\e(z_0) }{\rho^{d+1}} (1- \bar{z} H_{\Sigma}(x_0')) \, d \bar{z}  d \bar{y}
    \\ & + \int_{\mathcal{C}_\nu} \frac{w_\e'(z_0 + \bar{z}) f(z_0, \bar{y})}{\rho^{d+1}} (1- \bar{z} H_{\Sigma}(x_0')) \, d\bar{z} d \bar{y}
    \\ & + \int_{\mathcal{C}_\nu} \frac{w_\e(z_0 + \bar{z} + f(z_0, \bar{y}) ) - w_\e(z_0+ \bar{z}) - w_\e'(z_0+\bar{z}) f(z_0, \bar{y}) }{\rho^{d+1}} (1- \bar{z} H_{\Sigma}(x_0'))\, d\bar{z} d \bar{y} 
    \\ & + O \left( \int_{\mathcal{C}_\nu} \frac{\abs{w_\e(z_0+ \bar{z} + f(z_0, \bar{y}) )- w_\e(z_0) }}{\rho^{d+1}} ( \abs{z_0} \abs{\bar{y}}  + \rho^2 ) \, d\bar{z}  d \bar{y} \right). 
    \\ & = \lim_{\nu \to 0^+} I_1^\nu + I_2^\nu + I_3^\nu + O(I_4^\nu), \label{eq: expansion of inner contribution}
\end{align}
where $f(z_0, \overline{y})$ satisfies \eqref{eq: formula f}. To conclude the proof, we estimate separately the four terms. For the reader's convenience, we postpone these computations to \cref{ss: estimates of the four integrals}. Thus, by \eqref{eq: outer contribution}, \eqref{eq: expansion of inner contribution}, \cref{l: expansion for I_1}, \cref{l: expansion of I_2}, \cref{l: expansion of I_3}, \cref{l: expansion of I_4}, we infer that
\begin{equation}
    \gamma_{d, \sfrac{1}{2}}\lim_{\nu \to 0^+}  \Delta^{\frac{1}{2}}_\nu u_\e(x_0) = - (-\Delta)^\frac{1}{2} w_\e (z_0) + \frac{\gamma_{1,\sfrac{1}{2}}}{2 }H_{\Sigma} (x_0') \int_{-\sfrac{\delta}{\Lambda}}^{\sfrac{\delta}{\Lambda}} w_\e'(z_0+\bar{z}) \log(\abs{z}) \, d\bar{z} + \mathcal{R}_{\Lambda, \e}(x_0), 
\end{equation}
where $\mathcal{R}_{\Lambda,\e} \colon \Sigma_{\sfrac{\delta}{10 \Lambda}} \to \R$ is a bounded function satisfying \eqref{eq: bound reminder s=1/2}. Then, the proof is concluded.
\end{proof}

\subsection{Estimates of the four integrals} \label{ss: estimates of the four integrals} 

To conclude the proof of \cref{t:fractional laplacian} in the case $s= \sfrac{1}{2}$, we estimate the terms $I_1^\nu, I_2^\nu, I_3^\nu, I_4^\nu$ in \eqref{eq: expansion of inner contribution}. Throughout this section, we tacitly assume $\e \in (0,1)$ and $\Lambda \geq \Lambda_0$, where $\Lambda_0$ a purely geometric constant given by \cite{CFI24}*{Lemma 2.17}. We neglect constants $C(d,s,W,\delta, \Sigma)>0$, whereas it is crucial to keep the dependence on~$\e, \Lambda$ explicit. We sketch the argument following the lines of \cite{CFI24}*{Section 4.1}. We highlight the main differences and leave the remaining computations to the interested reader. To begin, we need a preliminary lemma. The proof relies on an explicit computation analogous to that of \cite{CFI24}*{Lemma 4.2}. 

\begin{lemma} \label{l:formula eta}
For any $z_0 \in \R, \ell, \e>0$ it holds
\begin{equation} \label{eq:formula eta 2}
\int_{-\ell}^\ell \frac{w_\e(z_0+z)-w_\e(z_0)}{\abs{z}^{2}} z \, dz = (w_\e(\ell+z_0) - w_\e(z_0-\ell)) \log(\ell) -  \int_{-\ell}^\ell w_\e'(z_0+z)\log(\abs{z}) \, dz. 
\end{equation}
\end{lemma}

In the following lemma, we estimate the term $I_1^\nu$ in \eqref{eq: expansion of inner contribution}. See \cite{CFI24}*{Lemma 4.3} for the proof.  

\begin{lemma} \label{l: expansion for I_1}
Let $I_1^\nu$ be given by \eqref{eq: expansion of inner contribution}. It holds 
\begin{align}
    \sup_{x_0 \in \Sigma_{\sfrac{\delta}{10 \Lambda}}} \abs{ \lim_{\nu \to 0^+} I_1^\nu + \frac{1}{\gamma_{d,\sfrac{1}{2}}}(-\Delta)^\frac{1}{2} w_\e(z_0) - \frac{\gamma_{1,\sfrac{1}{2}}}{\gamma_{d,\sfrac{1}{2}}} H_{\Sigma}(x_0') \int_{-\sfrac{\delta}{\Lambda}}^{\sfrac{\delta}{\Lambda}}w_\e'(z_0+\bar{z}) \log(\abs{\bar{z}}) \, d\bar{z} } \lesssim \Lambda.  
\end{align}
\end{lemma}

We also recall the following integral computation. 

\begin{lemma} \label{l:expansion log}
For any $\delta \geq 1$ it holds 
\begin{equation} \label{eq: log expansion 1}
    \int_{B_{\delta}^{d-1}} \frac{\abs{y_1}^2}{(1+\abs{y}^2)^{\frac{d+1}{2}}} \, dy = \frac{\gamma_{1, \sfrac{1}{2}}}{\gamma_{d,\sfrac{1}{2}}} \log(\delta) + O(1),  
\end{equation}
\begin{equation} \label{eq: log expansion 2}
    \int_{B_\delta^{d-1}} \frac{\abs{y}^\alpha}{(1+\abs{y}^2)^{\frac{d+1+\alpha}{2}}} \, dy = \frac{(d-1) \gamma_{1, \sfrac{1}{2}}}{\gamma_{d,\sfrac{1}{2}}} \log(\delta) +O(1), \qquad \alpha \geq 0.     
\end{equation}

\end{lemma}

\begin{proof}
We sketch the proof. Given $\delta \geq 1$, using \eqref{eq: constant fractional laplacian} and the properties of the $\Gamma$ function (see e.g. \cite{CFI24}*{Equation 7.1}) we compute in polar coordinates
\begin{align}
    \int_{B_\delta^{d-1}} \frac{\abs{y_1}^2}{(1+\abs{y}^2)^{\frac{d+1}{2}}} \, dy & = \frac{1}{d-1} \int_{B_{\delta}^{d-1}} \frac{\abs{y}^2}{(1+\abs{y}^2)^{\frac{d+1}{2}}} \, dy = \frac{\gamma_{1,\sfrac{1}{2}}}{\gamma_{d, \sfrac{1}{2}}} \int_0^{\delta} \frac{r^d}{(1+r^2)^{\frac{d+1}{2}}} \, dr. 
\end{align}
To conclude, we claim that 
\begin{equation} \label{eq: log expansion}
    \int_1^\delta \frac{r^d}{(1+r^2)^{\frac{d+1}{2}}} \, dr = \log(\delta) + O(1). 
\end{equation}
Indeed, we write
\begin{align}
    \abs{\int_1^\delta \frac{r^d}{(1+r^2)^{\frac{d+1}{2}}} \, dr - \log(\delta) } & = \abs{\int_1^\delta \left[ \frac{r^d}{(1+r^2)^{\frac{d+1}{2}}} - \frac{1}{r} \right] \, dr } \leq  \int_1^\delta \abs{\frac{r^{d+1} - (1+ r^2)^{\frac{d+1}{2}}}{(1+r^2)^{\frac{d+1}{2}} r}} \, dr
    \\ & = \int_1^\delta \abs{ \frac{r^{2d+2} -(1+r^2)^{d+1} }{r (1+r^2)^{\frac{d+1}{2}} \left(r + (1+ r^2)^{\frac{d+1}{2}}\right) } } \, dr 
    \leq \int_1^\delta \frac{r^{2d+1} + O(r^{2d}) }{r^{2d+3} } \, dr.
\end{align}
The latter is bounded independently of $\delta \geq 1$, thus proving \eqref{eq: log expansion 1}. The same argument yields \eqref{eq: log expansion 2}. 
\end{proof}

In the following lemma, we estimate the term $I_2^\nu$, as in \cite{CFI24}*{Lemma 4.4}.

\begin{lemma} \label{l: expansion of I_2}
Let $I_2^\nu$ be defined by \eqref{eq: expansion of inner contribution}. It holds that 
\begin{equation}
    \sup_{x_0 \in \Sigma_{\sfrac{\delta}{10 \Lambda}}} \abs{ \lim_{\nu \to 0^+} I_2^\nu + H_{\Sigma}(x_0') \frac{\gamma_{1,\sfrac{1}{2}}}{2 \gamma_{d,\sfrac{1}{2}}}  \int_{-\sfrac{\delta}{\Lambda}}^{\sfrac{\delta}{\Lambda}} w'_\e(z_0+\bar{z}) \log(\abs{\bar{z}}) \, d \bar{z} } \lesssim \log(\Lambda).   
\end{equation}
\end{lemma}

\begin{proof}
For simplicity, we denote by $\ell = \sfrac{\delta}{\Lambda}$. Then, we have that
\begin{align}
    \lim_{\nu \to 0^+} I_2^\nu & = \int_{B_\ell^{d-1}} \int_{-\ell}^\ell \frac{w_\e'(z_0+ \bar{z}) \sum_{i=1}^{d-1} k_i  \bar{y}_i^2 }{2 \rho^{d+1}} (1- \bar{z} H_{\Sigma}(x_0')) \, d \bar{z} d \bar{y} 
    \\ & \quad + O \left( \abs{z_0} \int_{B_\ell^{d-1}} \int_{-\ell}^\ell \frac{w'_\e(z_0 + \bar{z})}{\rho^{d+1}} \abs{\bar{y}}^2 \, d \bar{z}  d \bar{y} \right) + O \left( \int_{B_\ell^{d-1}} \int_{-\ell}^\ell \frac{w'_\e(z_0 + \bar{z})}{\rho^{d+1}} \abs{\bar{y}}^3 \, d \bar{z} d \bar{y} \right)
    \\ & = I_{2,1} + O\left(I_{2,2}\right) + O\left(I_{2,3}\right) 
\end{align}
and we estimate separately each term. 

\textsc{Estimate of $I_{2,1}$}. By \cref{l:expansion log} and the fundamental theorem of calculus, we have 
\begin{align}
    I_{2,1} & = \frac{H_\Sigma(x_0')}{2} \int_{-\ell}^\ell w_\e'(\bar{z}+z_0)(1- \bar{z} H_\Sigma (x_0')) \int_{B_\ell^{d-1}} \frac{\bar{y}_1^2}{(\abs{\bar{z}^2} + \abs{\bar{y}^2})^{\frac{d+1}{2}}} \, d\bar{y} d \bar{z} 
    \\ & = H_\Sigma(x_0') \frac{\gamma_{1,\sfrac{1}{2}}}{2 \gamma_{d, \sfrac{1}{2}}} \int_{-\ell}^\ell w_\e'(\bar{z}+z_0)(1- \bar{z} H_\Sigma (x_0')) \left[ \log \left(\frac{\ell}{\abs{\bar{z}}}\right) + O(1) \right] \, d \bar{z} 
    \\ & = - H_\Sigma(x_0') \frac{\gamma_{1,\sfrac{1}{2}}}{2 \gamma_{d, \sfrac{1}{2}}} \int_{-\ell}^\ell w_\e'(\bar{z}+z_0) \log(\abs{z}) \, dz + O(\abs{\log(\ell)}). 
\end{align}

\textsc{Estimate of $I_{2,2}$}. By \cref{l:expansion log} we have 
\begin{align}
    \abs{I_{2,2}} & = \abs{z_0} \int_{-\ell}^{\ell} w'_\e(z_0+\bar{z}) \int_{B_{\ell}^{d-1}}   \frac{\abs{\bar{y}}^2}{(\abs{\bar{z}}^2 + \abs{\bar{y}}^2)^{\frac{d+1}{2}}} \, d\bar{y} d\bar{z} 
    \\ & \lesssim \abs{z_0} \int_{-\ell}^\ell w_\e'(z_0+\bar{z}) \log\left(\frac{\ell}{\abs{\bar{z}}}\right) \, d\bar{z} + \abs{z_0} \int_{-\ell}^\ell w_\e'(z_0 + \bar{z}) \, d \bar{z} = A+B.
\end{align}
By the fundamental theorem of calculus, we find  
$$B = \abs{z_0} \int_{\R} w_\e'(z_0 + \bar{z}) \, d \bar{z} \lesssim \ell. $$
To bound the term $A$, we split 
\begin{equation}
    A =  \abs{z_0} \left(\int_{\abs{\bar{z}} \leq \sfrac{\abs{z_0}}{2}} + \int_{ \sfrac{\abs{z_0}}{2} \leq \abs{\bar{z}} \leq \ell}  \right) w_\e'(z_0+\bar{z}) \log \left( \frac{\ell}{\abs{\bar{z}}} \right) \, d \bar{z} = A_1 + A_2. 
\end{equation}
Then, by \eqref{eq: decay w'}, we see that 
\begin{equation}
    A_1 \lesssim \abs{z_0} \sup_{\abs{\zeta} \geq \sfrac{\abs{z_0}}{2}} w_\e'(\zeta) \int_{\abs{\bar{z}} \leq \ell} \log \left( \frac{\ell}{\abs{\bar{z}}} \right) \, d \bar{z} \lesssim \frac{\abs{z_0} \e^{-1} }{1+ \abs{z_0}^2 \e^{-2}} \ell = O(\ell),  
\end{equation}
since the function $t\mapsto \frac{\abs{t}}{1+t^2}$ is bounded in $\R$. Similarly, we obtain that
\begin{equation}
    A_2 \lesssim \abs{z_0} \log \left( \frac{2 \ell}{\abs{z_0}} \right) \int_{\R} w_\e'(z_0 + \bar{z}) \, d \bar{z} = O(\ell \abs{\log(\ell) }  ). 
\end{equation}

\textsc{Estimate of $I_{2,3}$}. We compute  
\begin{align}
    \abs{I_{2,3}} & \leq \int_{-\ell}^\ell w_\e'(z_0+\bar{z}) \int_{B_\ell^{d-1}} \frac{1}{\abs{\bar{y}}^{d-2}} \, d\bar{y} \lesssim \ell \int_{-\ell}^\ell w_\e'(z_0 + \bar{z}) \, d \bar{z} \lesssim \ell. 
\end{align} 

\end{proof}

We estimate the term $I_3^\nu$ as in \cite{CFI24}*{Lemma 4.5}. 

\begin{lemma} \label{l: expansion of I_3}
Let $I_3^\nu$ be defined by \eqref{eq: expansion of inner contribution}. It holds that 
\begin{equation}
    \sup_{x_0 \in \Sigma_{\sfrac{\delta}{10 \Lambda}} } \limsup_{\nu \to 0^+} \abs{I_3^\nu} \lesssim \e + \Lambda^{-1}  
\end{equation}
\end{lemma}

\begin{proof}
Let $\ell = \sfrac{\delta}{\Lambda}$. As in \cite{CFI24}*{Lemma 4.5}, we have 
\begin{align}
    \lim_{\nu \to 0^+} I_3^\nu & = \int_0^1 \int_{B_{\ell}^{d-1}} \bigg[  \left[ w'_\e(z_0 + \bar{z} + t f(\bar{y}, z_0)) \frac{1- \bar{z} H_{\Sigma}(x_0')}{\rho^{d+1}} \right]_{\bar{z} = -\ell}^{\bar{z}= \ell} 
    \\ & \qquad - \left[ w_{\e}(z_0+ \bar{z} + t f(\bar{y}, z_0) ) \frac{d}{d \bar{z}} \left( \frac{1- \bar{z} H_{\Sigma}(x_0')}{\rho^{d+1}} \right) \right]_{\bar{z}=-\ell}^{\bar{z} = \ell}
    \\ & \qquad + \int_{-\ell}^{\ell} w_\e(z_0 + \bar{z} + t f(\bar{y}, z_0)) \frac{d^2}{d \bar{z}^2} \left( \frac{1- \bar{z} H_{\Sigma}(x_0')}{\rho^{d+1}} \right)  \, d \bar{z} \bigg] f(\bar{y}, z_0)^2 \, d \bar{y}  dt
    \\ & = I_{3,1} + I_{3,2} + I_{3,3}. 
\end{align}

\textsc{Estimate of $I_{3,1}$ and $I_{3,2}$}. As shown in the proof of  \cite{CFI24}*{Lemma 4.5}, it is readily checked that\footnote{The authors found a typo in the proof of \cite{CFI24}*{Lemma 4.5}, which does not affect the validity of the argument in any sense. Indeed, the term $I_{3,2}$ is $ O(
\ell^{2-2s})$ instead of $\ell^{3-2s}$.} 
$$\abs{I_{3,1}} \lesssim \e, \qquad \abs{I_{3,2}} \lesssim \ell^{2-2s}. $$

\textsc{Estimate of $I_{3,3}$}. By an explicit computation, we see
\begin{equation} 
    \abs{\frac{d^2}{d \bar{z}^2} \left( \frac{1- \bar{z} H_{\Sigma}(x_0')}{\rho^{d+1}} \right) } \lesssim \frac{1}{\rho^{d+3}}.  \label{eq:est second derivative} 
\end{equation}
Using \eqref{eq:est second derivative}, \cref{l:expansion log} and recalling that $f(\bar{y}, z_0)^2 = O (\abs{\bar{y}}^4)$ (see \eqref{eq: formula f}), we compute 
\begin{align}
    \abs{I_{3,3}} & \lesssim \int_{-\ell}^{\ell} \int_{B_{\ell}^{d-1}} \frac{\abs{\bar{y}}^4}{(\abs{\bar{z}}^2 + \abs{\bar{y}}^2)^{\frac{d+3}{2}}} \, d\bar{y} d\bar{z}  \lesssim \int_{-\ell}^\ell \left[ \log\left( \frac{\ell}{\abs{\bar{z}}} \right) + O(1)\right] \, d \bar{z} = O(\ell).
\end{align}
\end{proof}

We estimate the term $I_4^\nu$ as in \cite{CFI24}*{Lemma 4.6}. 

\begin{lemma} \label{l: expansion of I_4}
Let $I_4^\nu$ be defined by \eqref{eq: expansion of inner contribution}. It holds that 
\begin{equation}
    \sup_{x_0 \in \Sigma_{\sfrac{\delta}{10 \Lambda}}} \limsup_{\nu \to 0^+} \abs{I_4^\nu} \lesssim  \Lambda^{-1}. 
\end{equation}

\end{lemma}

\begin{proof}
We define $\ell = \sfrac{\delta}{\Lambda}$ and we see that 
\begin{align}
    I_4 & \leq \abs{z_0}  \int_{B_\ell^{d-1}} \int_{-\ell}^\ell  \frac{\abs{w_\e(z_0 + \bar{z} + f(\bar{y}, z_0)) - w_\e(z_0+ \bar{z}) }}{\rho^{d}}  \, d \bar{z}  d \bar{y} 
    \\ & \qquad + \abs{z_0}  \int_{B_\ell^{d-1}} \int_{-\ell}^\ell  \frac{\abs{w_\e(z_0 + \bar{z}) - w_\e(z_0) }}{\rho^{d}}  \, d \bar{z}  d \bar{y} 
    \\ &  \qquad + \int_{B_\ell^{d-1}} \int_{-\ell}^\ell  \frac{\abs{w_\e(z_0 + \bar{z} + f(\bar{y}, z_0)) - w_\e(z_0) }}{\rho^{d-1}}  \, d \bar{z}  d \bar{y}
    \\ & = I_{4,1}+ I_{4,2}+ I_{4,3}. 
\end{align}

\textsc{Estimate of $I_{4,1}$ and $I_{4,3}$}. As shown in the proof of \cite{CFI24}*{Lemma 4.6} we have\footnote{The authors found a typo in the proof of \cite{CFI24}*{Lemma 4.6}, which does not affect the validity of the argument in any sense. Indeed the term $I_{4,1}$ is  $O(\ell^{3-2s})$ instead of $\ell^{2-2s}$. } 
$$ \abs{I_{4,1}} + \abs{I_{4,3}} \lesssim \ell^{3-2s} + \ell^{2-2s}. $$

\textsc{Estimate of $I_{4,2}$}. The estimate of $I_{4,2}$ is similar to that of $I_{2,2}$ in \cref{l: expansion of I_2}. Indeed, by the fundamental theorem of calculus, we have 
\begin{align}
    I_{4,2} & \leq \abs{z_0} \int_0^1 \int_{-\ell}^\ell w'_\e(z_0+t\bar{z}) \int_{B_\ell^{d-1}} \frac{\abs{\bar{z}}}{\rho^{d}} \, d \bar{y}  d \bar{z}  dt 
    \\ & \lesssim \abs{z_0} \int_0^1 \frac{1}{t} \int_{-\ell t}^{ \ell t} w_\e'(z_0 + z) \, d z  dt 
    \\ & = \abs{z_0} \left( \int_0^{\sfrac{\abs{z_0}}{2\ell}} + \int_{\sfrac{\abs{z_0}}{2 \ell}}^1 \right) \frac{w_\e(z_0 + \ell t) - w_\e(z_0-\ell t)}{t} \, dt = \bar{A} + \bar{B}. 
\end{align}
Since $\abs{w_\e} \leq 1$ and $\abs{z_0} \leq \sfrac{\ell}{10}$, it is readily checked that $\abs{\bar{B}} \lesssim \ell$. To estimate $\bar{A}$, using \eqref{eq: decay w'} and since $\abs{z_0 \pm \ell t} \geq \sfrac{\abs{z_0}}{2}$ for $t \in \left[ 0, \sfrac{\abs{z_0}}{2\ell} \right]$, we compute 
\begin{equation}
    \abs{\bar{A}} \lesssim \abs{z_0} \int_0^{\sfrac{\abs{z_0}}{2\ell}} \frac{1}{t} \frac{\abs{t \ell} \e^{-1}} {1 + \abs{z_0}^2 \e^{-2}} \, dt \lesssim \ell.  
\end{equation}

\end{proof}

\section{Proof of Theorem~\ref{t:main}} \label{s:proof of main theorem}

The following is the extension of \cite{CFI24}*{Proposition 5.3} to the regime $s \in \left[ \sfrac{1}{2}, \sfrac{3}{4}\right)$. 

\begin{lemma} \label{p: constant willmore}
Let $s \in \left(\sfrac{1}{2}, \sfrac{3}{4}\right)$. For any $\delta >0$ it holds 
\begin{equation}
    \lim_{\ell \to 0^+} \lim_{\e \to 0^+}  \int_{-\ell}^\ell \left( \int_{-\delta}^\delta \frac{w_\e'(z_0+ \bar{z})}{\abs{\bar{z}}^{2s-1}} \, d \bar{z} \right)^2 \, d z_0  = 0. 
\end{equation} 
For $s = \sfrac{1}{2}$ for any $\delta>0$ it holds 
\begin{equation}
    \lim_{\ell \to 0^+} \lim_{\e \to 0^+}  \int_{-\ell}^\ell \left( \int_{-\delta}^\delta w_\e'(z_0 + \bar{z}) \abs{\log(\abs{z})} \, d\bar{z} \right)^2 \, d z_0 = 0. 
\end{equation}  
\end{lemma}

\begin{proof}
Given~$\delta, \ell, \e >0$ and $s\in \left(\sfrac{1}{2}, \sfrac{3}{4} \right)$, we have that
\begin{align}
     \int_{-\ell}^\ell \left( \int_{-\delta}^\delta \frac{w_\e'(z_0+ \bar{z})}{\abs{\bar{z}}^{2s-1}} \, d \bar{z} \right)^2 \, d z_0 \lesssim \int_{\ell}^{\ell} \int_{-\delta}^{\delta} \frac{w_\e'(z_0+ \bar{z})}{\abs{\bar{z}}^{4s-2}}\, d \bar{z}  dz_0 = \int_{-\delta}^\delta \frac{w_\e(\bar{z} + \ell) - w_\e(\bar{z} - \ell)}{\abs{\bar{z}}^{4s-2}}\, d \bar{z} 
\end{align}
Since $4s-2 < 1$, by dominated convergence, we find 
\begin{equation}
    \lim_{\ell \to 0^+} \lim_{\e \to 0^+} \int_{-\delta}^\delta \frac{w_\e(\bar{z} + \ell) - w_\e(\bar{z} - \ell)}{\abs{\bar{z}}^{4s-2}}\, d \bar{z} = \lim_{\ell \to 0^+} \int_{-\delta}^{\delta} \frac{\sgn(\bar{z} + \ell) - \sgn(\bar{z} - \ell)}{\abs{\bar{z}}^{4s-2}}\, d \bar{z} = 0. 
\end{equation}
If $s = \sfrac{1}{2}$, the proof is very similar, since $\abs{\log(\abs{\cdot})}^2$ is integrable at the origin. 
\end{proof}

In the following result, we study the asymptotic behavior of the energy $\G_{s,\e}$ on the composition of the optimal profile with the modified signed distance. The proof is based on the expansion of the fractional Laplacian established in \cref{t:fractional laplacian} as well as \cref{l: limit W'(u_e)}, \cref{prop:N-delta} and \cref{p: constant willmore}. 

\begin{proposition}
 \label{prop:main} 
Let $E$ be a bounded open set of class $C^3$. Let $s \in \left(0,\sfrac{3}{4}\right)$ and
consider a bounded open set $\Omega$. Then, there exists $\delta_0>0$ depending only on $\Sigma = \partial E$ with the following property. For any $\delta < \delta_0$ and for any $\eta \in \mathcal{K}_\delta$ given by \cref{d:regular distance}, we define $u_\e^{\eta, \delta}$ as in \eqref{eq: recovery sequence}. Then, we have
\begin{equation} \label{eq:energy of recovery}
\lim_{\e \to 0^+} \G_{s,\e} (u_\e^{\eta, \delta}, \Omega) = \int_{\Omega} \Bigg( (-\Delta)^s \chi_E -\frac{\gamma_{1,s}}{s} \frac{\chi_E}{\abs{\beta_\Sigma^{\eta, \delta}}^{2s}} \Bigg)^2 dx. 
\end{equation}
\end{proposition}

\begin{proof}
We focus on the case $s \in \left(\sfrac{1}{2}, \sfrac{3}{4}\right)$. The case $s \leq \sfrac{1}{2}$ is completely analogous, thus we leave it to the interested reader. We point out that it suffice to use the expansion of the fractional Laplacian given by \cref{t:fractional laplacian} instead of \cite{CFI24}*{Theorem 4.1}. From now on, we fix $s \in \left(\sfrac{1}{2}, \sfrac{3}{4}\right)$. We also pick
a geometric parameter~$\delta_0 >0$, depending only on $\Sigma$, such that $\d_\Sigma \in C^3 (\Sigma_{\delta_0})$ (see \cref{l: regularity of distance function}). For $\delta < \delta_0$ and $\eta \in \mathcal{K}_\delta$ as in \cref{d:regular distance}, we define $u_\e^{\eta,\delta}$ by \eqref{eq: recovery sequence}. Even if $u_\e^{\eta, \delta}$ explicitly depends on $\e, \delta, \eta$, we suppress the superscript $\eta,\delta$ for simplicity. Since $\eta$ has the same sign of $\d_\Sigma$ and $\min_{x \in \Sigma_\delta^c} \abs{\eta(x)} >0$, we have that $u_\e \to \chi_E $ in $L^1_{loc}(\R^d)$. Then, given $\ell < \sfrac{\delta}{10}$, we compute 
\begin{equation}
    \mathcal{G}_{s, \e}(u_\e, \Omega) = \left[ \int_{\Sigma_\ell \cap \Omega} + \int_{\Sigma_\ell^c \cap \Omega} \right] \left( (-\Delta)^s u_\e + \frac{W'(u_\e)}{\e^{2s}} \right)^2 \, dx = I_{\e,\ell, \eta, \delta} + II_{\e, \ell, \eta, \delta}.  
\end{equation}

\textsc{\underline{The energy at the interface:}} To estimate the energy in $\Sigma_\ell \cap \Omega$, we use the expansion of the fractional Laplacian in $\Sigma_{\delta}$ given by \cref{t:fractional laplacian}\footnote{Here, we pick $\Lambda = \Lambda_0$ and consequently we need $\ell \leq \sfrac{\delta}{10 \Lambda_0} $.}. Therefore, we have 
\begin{align}
    I_{\e, \ell, \eta, \delta} & = \int_{\Sigma_\ell \cap \Omega} \left( \frac{\gamma_{1,s}}{2(2s-1)} H_\Sigma(x') \int_{-\delta}^\delta \frac{w_\e'(z+ \bar{z})}{\abs{z}^{2s-1}} \, d \bar{z}  + R^{\eta, \delta}_{\e, \Lambda_0}(x) \right)^2 \, dx \leq A + B,   
\end{align}
where we set $x' = \pi_\Sigma(x)$ and $z = \d_\Sigma(x)$. Since the reminder satisfies \eqref{eq: bound reminder}, we infer 
$$\lim_{\ell \to 0^+} \limsup_{\e \to 0^+} B =0. $$
The behavior of the term $A$ dictates the difference with the case $s \geq \sfrac{3}{4}$. By the coarea formula, we see that 
\begin{align}
    A \leq \frac{\gamma_{1,s}^2}{4(2s-1)^2} \left( \int_{\Sigma} H_\Sigma(x')^2 \, d \mathcal{H}^{d-1}(x') \right)^2 \int_{-\ell}^\ell \left( \int_{-\delta}^\delta \frac{w_\e'(z+ \bar{z})}{\abs{\bar{z}}^{2s-1}} \, d\bar{z} \right)^2 \, dz. 
\end{align}

By \cref{p: constant willmore}, it follows  
$$\lim_{\ell \to 0^+} \lim_{\e \to 0^+} A = 0 \qquad \forall \delta < \delta_0. $$
To summarize, we have that 
\begin{equation}
    \lim_{\ell \to 0^+} \lim_{\e \to 0^+} I_{\e, \ell, \eta, \delta} = 0 \qquad \forall \delta < \delta_0, \, \forall \eta \in \mathcal{K}_\delta. \label{eq: lim first term}
\end{equation}

\textsc{\underline{The energy away from the interface:}} To estimate the energy in $\Sigma_\ell^c \cap \Omega$, by \cref{l: limit W'(u_e)} and recalling that $\inf_{x \in \Sigma_\ell^c} \abs{ \beta_\Sigma^{\eta, \delta}} >0 $, we have that 
$$\lim_{\e \to 0^+} \frac{W'(u_\e(x))}{\e^{2s}} = - \frac{\gamma_{1,s}}{s} \frac{\chi_E(x)}{\abs{\beta_\Sigma^{\eta, \delta} (x) }^{2s}}$$
uniformly on $\Sigma_\ell^c$. By explicit computations and exploiting the decay of $w', w''$ (see \cref{t:optimal profile}), one can check that $u_\e \to \chi_E$ in $C^2(\Sigma_{\sfrac{\ell}{2}}^c)$. Since the sequence $\{ u_\e\}_\e$ is bounded in $L^\infty(\R^d)$, it follows that $(-\Delta)^s u_\e \to (-\Delta)^s \chi_E$ uniformly on $\Sigma_\ell^c$. Thus, we conclude that 
\begin{align}
    \lim_{\e \to 0^+}  II_{\e, \ell, \eta, \delta} = \int_{\Sigma_\ell^c \cap \Omega} \Bigg( (-\Delta)^s \chi_E - \frac{\gamma_{1,s}}{s} \frac{\chi_E}{\abs{\beta_\Sigma^{\eta, \delta}}^{2s}} \Bigg)^2 \, dx. \label{eq:lim second term}
\end{align}
Then, \eqref{eq:energy of recovery} follows by \eqref{eq: lim first term} and by taking the limit as $\ell \to 0$ in \eqref{eq:lim second term}. 
\end{proof}

Lastly, we show how to optimize with respect to the choice of $\delta, \eta$. 

\begin{corollary} \label{cor:optimization}
Under the assumptions of \cref{prop:main}, it holds 
\begin{equation} \label{eq:optimization eta delta}
    \lim_{\delta \to 0^+} \inf_{\eta \in \mathcal{K}_\delta} \lim_{\e \to 0^+} \G_{s,\e}(u_\e^{\eta,\delta}, \Omega ) = 0. 
\end{equation}
\end{corollary}

\begin{proof}

By \eqref{eq: recovery sequence} and recalling the definition of $\beta_\Sigma^{\eta, \delta}$ (see \eqref{eq: definition beta}), we have 
\begin{align}
    \lim_{\e \to 0^+} \mathcal{G}_{s, \e} (u_\e^{\eta, \delta}, \Omega) &  = \int_{\Sigma_\delta \cap \Omega} \bigg( (-\Delta)^s \chi_E - \frac{\gamma_{1,s}}{s} \frac{\chi_E}{\abs{\d_\Sigma}^{2s}} \bigg)^2 \, dx + \int_{\Sigma_\delta^c \cap \Omega} \bigg( (-\Delta)^s \chi_E - \frac{\gamma_{1,s}}{s} \frac{\chi_E}{\abs{\eta}^{2s}} \bigg)^2 \, dx  
    \\ & = A_\delta + B_{\eta, \delta}.    \label{eq: before optimization} 
\end{align}

\textsc{\underline{Optimizing with respect to $\eta$:}} Clearly, $A_\delta$ is independent of $\eta$. We split $B_{\eta, \delta}$ into 
\begin{align}
    B_{\eta, \delta} & = \int_{\Sigma_\delta^c \cap \Omega \cap E} \left( (-\Delta)^s \chi_E - \frac{\gamma_{1,s}}{s} \frac{1}{\eta(x)^{2s}} \right)^2 \, dx + \int_{ \Sigma_\delta^c \cap \Omega \cap E^c } \left( (-\Delta)^s \chi_E + \frac{\gamma_{1,s}}{s} \frac{1}{\abs{\eta(x)}^{2s}} \right)^2 \, dx 
    \\ & = C_{1,\eta, \delta} + C_{2,\eta, \delta}. 
\end{align}
We claim that 
\begin{equation}
    \inf_{\eta \in \mathcal{K}_\delta} B_{\delta, \eta} = 0 \qquad \forall \delta < \delta_0. 
\end{equation}
Therefore, the contribution away from the boundary becomes negligible by a suitable choice of the cut off. This is possible because we have to minimize a functional that does not involve spatial derivatives. More precisely, computing explicitly the fractional Laplacian, we observe that $(-\Delta)^s \chi_E \in C^\infty_{loc}(\R^d \setminus \Sigma)$, $(-\Delta)^s \chi_E$ has the same sign as $\chi_E$ on $\R^d \setminus \Sigma$ and $\inf_{x \in \Sigma_\delta^c} \abs{(-\Delta)^s \chi_E} >0$. Therefore, for any $\delta' \in \left(0, \sfrac{\delta_0}{2} \right)$, we find $\eta^{\delta'} \in \mathcal{K}_\delta$ such that 
\begin{equation}
\label{eq:def-eta}
    \eta^{\delta'} = \begin{cases}
        c_s ((-\Delta)^s \chi_E )^{-\frac{1}{2s}} & x \in E \setminus \Sigma_{\delta + \delta'}, 
        \\ - c_s \abs{(-\Delta)^s \chi_E}^{-\frac{1}{2s}} & x \in E^c \setminus \Sigma_{\delta + \delta'}, 
    \end{cases}
\end{equation} 
(for an explicit constant $c_s >0$) and such that 
$$ 0 < c_1(\delta,s) \leq \abs{\eta(x)} \leq c_2(\delta, s) \qquad \forall x \in \Sigma_{\delta + \delta'} \setminus \Sigma_\delta, $$
for some explicit constants $0<c_1< c_2 < \infty$ depending only on $s, \delta, \delta_0$. Thus, we conclude that 
\begin{equation}
    \inf_{\eta \in \mathcal{K}_{\delta}} B_{\eta,\delta}\leq \lim_{\delta' \to 0^+} C_{1, \eta^{\delta'}, \delta} + C_{2, \eta^{\delta'}, \delta} \lesssim \lim_{\delta' \to 0^+} \int_{\Sigma_{\delta + \delta'} \setminus \Sigma_{\delta}} \abs{\abs{(-\Delta)^s \chi_E} + c_1^{-2s}}^2 \, dx  = 0 \qquad \forall \delta >0.    
\end{equation}

\textsc{\underline{The limit $\delta \to 0$:}} To summarize, we have shown that 
\begin{equation}
    \inf_{\eta \in \mathcal{K}_\delta}  \lim_{\e \to 0^+} \mathcal{G}_{s, \e} (u_{\e}^{\eta, \delta})  \leq A_\delta = \int_{\Sigma_\delta} \bigg( (-\Delta)^s \chi_E - \frac{\gamma_{1,s}}{s} \frac{\chi_E}{\abs{\d_\Sigma}^{2s}} \bigg)^2 \, dx. 
\end{equation}
By \cref{prop:N-delta}, the latter goes to $0$ as $\delta \to 0$, thus showing \eqref{eq:optimization eta delta}. 
\end{proof}

\subsection{The Gamma-convergence result}

We now present the proof of \cref{t:main}. To begin, we point out that, by \cref{prop:main} and \cref{cor:optimization}, it follows that
\begin{equation}
    \Gamma - \limsup_{\e \to 0^+} \mathcal{G}_{s,\e} (\chi_E, \Omega) = 0. 
\end{equation} 
For the reader's convenience, we consider separately the regime $s \in \left[ \sfrac{1}{2}, \sfrac{3}{4}\right)$ and $s < \sfrac{1}{2}$. In the first case, the argument follows immediately by the Gamma-convergence result in \cite{SV12} and our \cref{prop:main} and \cref{cor:optimization}. The second case is more delicate.

\begin{proof}[Proof of \cref{t:main} for $s \in \left[ \sfrac{1}{2}, \sfrac{3}{4}\right)$] 
To begin, we focus on the case $s \geq \sfrac{1}{2}$. Let $E$ be a set with finite perimeter in $\Omega$. Up to passing to the complementary, we can assume that $E$ is bounded. The liminf inequality follows by \cite{SV12}*{Proposition 4.5} and the fact that the functional $\mathcal{G}_{s,\e}$ is nonnegative. 

To show the limsup inequality, we find a sequence of bounded open sets $\{E_n\}$ with smooth boundary converging to $E$ in $L^1$ (see e.g. \cites{M12, Modica}), intersecting $\partial \Omega$ transversally and such that 
\begin{equation}
    \lim_{n \to +\infty} \Per(E_n, \Omega) = \Per(E, \Omega).  
\end{equation} 
Thus, let $n \in \N$ and $\delta_n >0$ (small enough, depending on the geometry of $\Sigma_n$) such that for any $\d_{\Sigma_n}$ is smooth in $(\Sigma_n)_{\delta_n}$, as in \cref{l: regularity of distance function}. For any $\delta < \delta_n$ and $\eta \in \mathcal{K}_{\delta}$, we define 
$$u_{n,\e}(x):= w_\e (\beta_{\Sigma_n}^{\eta, \delta} (x)). $$
Then, $u_{\e,n}^{\eta, \delta} \to \chi_{E_n}$ as $\e \to 0$ in $L^1_{loc}(\R^d)$ and, combining the argument of \cite{SV12}*{Proposition 4.6} and \cref{prop:main}, we infer that 
\begin{equation}
    \limsup_{\e \to 0^+} \mathcal{F}_{s, \e}(u_{n,\e}^{\eta, \delta}, \Omega) + \mathcal{G}_{s, \e}(u_{n,\e}^{\eta, \delta}, \Omega) \leq c_\star \Per(E_n, \Omega) + \int_{\Omega} \Bigg( (-\Delta)^s \chi_{E_n} - \frac{\gamma_{1,s}}{s} \frac{\chi_E}{\abs{\beta_{\Sigma_n}^{\eta,\delta}}^{2s}} \Bigg)^2 dx ,  
\end{equation}
where $c_\star = c_\star(s, W) >0$ is an explicit constant. Accordingly, taking infimum with respect to $\eta$ and the limit as $\delta \to 0$ as in \cref{cor:optimization} and finally the limit as $n \to \infty$, we find a diagonal sequence $v_{\e_n} \to \chi_E $ in $L^1_{loc}(\R^d)$ such that 
\begin{equation}
    \limsup_{n \to +\infty} \mathcal{F}_{s, \e_n}(v_{\e_n}, \Omega) + \mathcal{G}_{s, \e_n}(v_{\e_n}, \Omega) \leq c_\star \Per(E, \Omega),
\end{equation}as desired.
\end{proof}

\begin{proof} [Proof of \cref{t:main} for $s < \sfrac{1}{2}$] 
Consider a bounded set $E$ with $\Per_{2s}(E, \Omega)< \infty$. To begin, we recall that the liminf inequality follows by Fatou's lemma (see e.g. the proof of \cite{SV12}*{Theorem 1.2}) and the fact that the functionals $\G_{s,\e}$ are nonnegative. To show the limsup inequality, by the same argument discussed in details in the case $s \in \left[ \sfrac{1}{2}, \sfrac{3}{4}\right)$ and the approximation result in \cref{l: approximation of fractional perimeter perimeter}, we may assume without loss of generality that $E$ is a bounded smooth set. More precisely, combining the Gamma-convergence result in \cite{SV12}, \cref{prop:main} and \cref{cor:optimization}, we only have to show that for any $\eta \in \mathcal{K}_\delta$ ($\delta < \delta_0$ sufficiently small), $u_\e:= w_\e (\beta_\Sigma^{\eta, \delta})$ can be used as a recovery sequence for the $2s$-fractional perimeter\footnote{For $s < \sfrac{1}{2}$ the constant recovery sequence $v_\e = \chi_E$ has been used by the authors of \cite{SV12} to show the limsup inequality. This motivates our analysis. }. For simplicity, we denote by 
$$\mathcal{K}_s(u_\e, \Omega) := \frac{\gamma_{d,s}}{4} \iint_{\R^d \times \R^d \setminus (\Omega^c \times \Omega^c) } \frac{\abs{u_\e(x) - u_\e(y)}^2}{\abs{x-y}^{d+2s}} \, dx dy $$
and for any Borel function $h$ and any Borel sets $A,B$ we let 
$$h(A,B) : = \int_A \int_{ B} \frac{\abs{h(x) - h(y)}^2}{\abs{x-y}^{d+2s}} \, dx dy. $$
Thus, given $\delta < \delta_0$ and $\eta \in \mathcal{K}_\delta$, we define $u_\e^{\eta, \delta}$ by \eqref{eq: recovery sequence} and prove that 
\begin{equation} \label{eq: recovery for fractional perimeter}
    \lim_{\e \to 0^+} \mathcal{K}_s(u_\e^{\eta, \delta}, \Omega) + \frac{1}{\e^{2s}} \int_{\Omega} W(u_\e^{\eta, \delta}) \, dx = \gamma_{d,s} \Per_{2s}(E, \Omega).   
\end{equation}
In the following, we omit the superscript $\eta, \delta$ and we neglect multiplicative constants possibly depending on $\Sigma, \eta, \delta, W, w, d,s$. 

\textsc{\underline{The potential energy}:} The term involving the potential $W$ can be estimated using the decay of the optimal profile. Pick a small parameter $\ell < \delta$. Then, by \eqref{eq: decay w'} we have 
\begin{align}
    \e^{-2s} \int_{\Sigma_\ell} W(u_\e) & = \e^{-2s} \int_{\Sigma_\ell} \abs{ W(u_\e) - W(\pm 1) } \, dx 
    \\ & \lesssim \e^{-2s} \int_{\Sigma_\ell} \frac{1}{1+ \d_\Sigma(x)^{2s}\e^{-2s}} \, dx \lesssim \int_{\Sigma_\ell} \abs{\d_{\Sigma}(x)}^{-2s} \, dx = O(\ell^{1-2s})
\end{align}
uniformly with respect to $\e$. Moreover, using again \eqref{eq: decay w'} we have 
\begin{align}
    \e^{-2s} \int_{\Omega \cap \Sigma_\ell^c} W(u_\e(x)) \, dx & = \e^{-2s} \int_{\Omega \cap \Sigma_\ell^c}  W(u_\e(x)) - W(\pm 1) - W'(\pm 1) (u_\e(x) - (\pm 1) )  
    \\ & \lesssim c(\ell) \e^{-2s} \int_{\Omega \cap \Sigma_\ell^c} \left[ \frac{1}{1+ \beta_\Sigma(x)^{2s} \e^{-2s} } \right]^2 \lesssim c(\ell) \e^{2s}. 
\end{align}
To summarize, we have 
$$\lim_{\e \to 0^+} \e^{-2s} \int_{\Omega} W(u_{\e}) \, dx = 0. $$

\textsc{\underline{Splitting the interaction energy}:} To treat $\mathcal{K}_s (u_\e, \Omega)$, by standard computations we have
\begin{align}
    \frac{4}{\gamma_{d,s}} & \mathcal{K}_s(u_\e, \Omega) =  u_\e(E \cap \Omega, E \cap \Omega) + u_\e(E^c \cap \Omega, E^c \cap \Omega) 
    \\ & + 2 \left[ u_\e(E \cap \Omega, E^c) +  u_\e(E \cap \Omega, E \cap \Omega^c) + u_\e(E^c \cap \Omega, E^c \cap \Omega^c ) + u_\e(E^c \cap \Omega, E\cap \Omega^c) \right].  
\end{align}
Then, we study separately each term.

\textsc{\underline{The inner-outer contributions}:} We claim that

\begin{equation} \label{eq claim 0}
    \lim_{\e \to 0^+}   u_\e(E \cap \Omega, E^c) = \int_{E\cap \Omega} \int_{E^c} \frac{4}{\abs{x-y}^{d+2s}} \, dx dy =  \chi_E (E \cap \Omega, E^c). 
\end{equation}
Indeed, recall that $\Per_{2s} (E, \Omega)< \infty$, i.e. $\abs{x-y}^{-d-2s} \in L^1(E \times E^c \setminus (\Omega^c \times \Omega^c))$, $\abs{u_\e(x) - u_{\e}(y)} \leq 2$ for $x \in E, y \in E^c$ and it converges pointwise to $2$ in $E \times E^c$. Then, \eqref{eq claim 0} follows by dominated convergence. Similarly, since $\Per_{2s}(\Omega, \R^d)$ is finite, we find that 
\begin{align}
    & \lim_{\e \to 0^+} u_\e(E \cap \Omega, E \cap \Omega^c) \leq \lim_{\e \to 0^+} u_\e(\Omega, \Omega^c) = 0,   
    \\[0.5ex] & \lim_{\e \to 0^+} u_\e(E^c \cap \Omega, E^c \cap \Omega^c) \leq \lim_{\e \to 0^+} u_\e(\Omega, \Omega^c) = 0,  
    \\[0.5ex] & \lim_{\e \to 0^+} u_\e(E^c \cap \Omega, E \cap \Omega^c) = \chi_E( E^c \cap \Omega, E \cap \Omega^c ). 
\end{align}
Then, we remark that 
\begin{equation}
    \chi_E (E\cap \Omega, E^c) + \chi_E(E^c\cap \Omega, E \cap \Omega^c) = \iint_{E\times E^c \setminus (\Omega^c \times \Omega^c) } \frac{4}{\abs{x-y}^{d+2s}} \, dx dy = 2 \Per_{2s} (E, \Omega). 
\end{equation}

\textsc{\underline{The inner-inner contribution}:} It remains to study $u_\e(E \cap \Omega, E\cap \Omega)$ and $u_\e(E^c \cap \Omega, E^c \cap \Omega)$. We focus on the first term, the second being analogous. Pick $ 0< \ell < \delta$ that is a regular value for the distance function. Then, we split the energy as follows
$$ u_\e(E \cap \Omega, E \cap \Omega) \leq u_\e(E, E) = u_\e(\Sigma_\ell \cap E, \Sigma_\ell \cap E) + 2 u_\e(\Sigma_\ell \cap E, \Sigma_\ell^c \cap E) + u_\e( \Sigma_\ell^c \cap E, \Sigma_\ell^c \cap E). $$
We claim that 
\begin{align}
    & \lim_{\e \to 0^+} u_{\e}(\Sigma_\ell^c \cap E, \Sigma_\ell^c \cap E) = 0, \label{eq claim 1}
    \\[0.5ex] & \lim_{\e \to 0^+} u_{\e} ( \Sigma_\ell \cap E, \Sigma_\ell^c \cap E) = 0, \label{eq claim 2}
    \\[0.5ex] & \lim_{\e \to 0^+} u_\e(\Sigma_\ell \cap E, \Sigma_\ell \cap E) =0 . \label{eq claim 3} 
\end{align}
To show \eqref{eq claim 1}, by the fundamental theorem of calculus and the decay of the optimal profile (see \cref{t:optimal profile}) we have 
\begin{align}
    u_{\e}(\Sigma_\ell^c \cap E, \Sigma_\ell^c \cap E) & \lesssim\int_{\Sigma_\ell^c \cap E} \int_{\Sigma_\ell^c \cap E} \frac{ \abs{\e^{-1} \abs{x-y} \left(\sfrac{\ell}{\e}\right)^{-2s-1}}^2}{\abs{x-y}^{d+2s}} \, dy dx
    \\ & \lesssim\ell^{-4s-2} \e^{4s} \int_{\Sigma_\ell^c \cap E} \int_{ \Sigma_\ell^c \cap E} \frac{1}{\abs{x-y}^{d+2s-2} }\, dx dy 
    \\ & \lesssim \ell^{-4s-2} \e^{4s} \int_{E \cap \Sigma_\ell^c} \int_{B_R(x)} \frac{1}{\abs{x-y}^{d+2s-2}} \, dy dx, 
\end{align}
for some $R>0$ large enough. The latter vanishes as $\e \to 0$. 
 
The proof of \eqref{eq claim 2} follows the line of that of \eqref{eq claim 0}, thus we leave it to the reader. It is enough to recall that the level set $\{\d_\Sigma = \ell\}$ is a $C^3$ hypersurface by the choice of $\ell$. 

\textsc{\underline{Proof of \eqref{eq claim 3}. One dimensional reduction}:}

By the coarea formula, we compute 
\begin{align}
    u_\e(\Sigma_\ell \cap E, \Sigma_\ell \cap E) = \int_0^\ell \, dt \int_0^\ell \, dz \frac{\abs{w_\e(t) -w_\e(z)}^2}{\abs{t-z}^{1+2s}} \int_{\d_\Sigma(x) =t} \int_{\d_{\Sigma}(y) = z} \frac{\abs{t-z}^{1+2s}}{\abs{x-y}^{d+2s}} \, d\mathcal{H}^{d-1} \, d \mathcal{H}^{d-1}.  
\end{align}
We claim that 
\begin{equation} \label{eq claim 4}
    \sup_{x_0 \in \Sigma_\ell, t \in (0,\ell)} \int_{\d_{\Sigma}(x) = t} \frac{\abs{\d_\Sigma(x_0) -t}^{1+2s}}{\abs{x_0-x}^{d+2s}} \, d \mathcal{H}^{d-1}(x) < \infty, 
\end{equation}
\begin{equation} \label{eq claim 5}
    \int_0^\ell  \int_0^\ell  \frac{\abs{w_\e(t) - w_\e(z)}^2}{\abs{t-z}^{1+2s}} \, dz dt \lesssim \e^{1-2s} + \e^{2s}. 
\end{equation}
Then, \eqref{eq claim 3} follows by \eqref{eq claim 4} and \eqref{eq claim 5}. 

\textsc{Proof of \eqref{eq claim 4}:} Let $x_0 \in \Sigma_\ell$ and $t \in (0,\ell)$. Let
also~$x_0' = \pi_{\Sigma}(x_0)$ be the point of minimal distance between $\Sigma$ and $x_0$. Without loss of generality, we may assume that $x_0' = 0, x_0 = z_0 e_d$, with $z_0= \d_\Sigma(x_0)$ and that $\text{Tan}_{x_0'}(\Sigma) = \text{Span}(e_1, \dots, e_{d-1})$. Then, with the notation of \cite{CFI24}*{Section 2.7}, we have 
\begin{align}
    \int_{\d_{\Sigma}(x)= t} \frac{1}{\abs{x_0-x}^{d+2s}} \, d\mathcal{H}^{d-1} & \lesssim 1 + \int_{\{\d_{\Sigma}(x) = t \} \cap B_{\ell}^d} \frac{1}{\abs{x-x_0}^{d+2s}} \, d\mathcal{H}^{d-1}
    \\ & \lesssim 1 + \int_{B_{\ell}^{d-1}} \frac{\abs{\det \nabla \Phi_t(y)}}{\abs{\Phi(t,y) - \Phi(0, z_0)}^{d+2s}} \, dy 
    \\ & \lesssim 1 + \int_{B_\ell^{d-1}} \frac{1}{(\abs{t-z_0}^2 + C\abs{y}^2)^{\frac{d+2s}{2}} } \, dy 
    \\ & \lesssim 1 + \frac{1}{\abs{t-z_0}^{1+2s}} \int_{\R^{d-1}} \frac{1}{(1 + C \abs{y}^2)^{\frac{d+2s}{2}}} \, dy 
    \\ & \lesssim 1 + \abs{t-\d_\Sigma(x_0)}^{-1-2s}.  
\end{align}

\textsc{Proof of \eqref{eq claim 5}: } We compute 
\begin{align}
    \int_0^\ell  \int_0^\ell  \frac{\abs{w_\e(t) - w_\e(z)}^2}{\abs{t-z}^{1+2s}} \, dz dt & = \e^{1-2s} \int_0^{\sfrac{\ell}{\e}} \int_0^{\sfrac{\ell}{\e}} \frac{\abs{w(t) -w(z) }^2}{\abs{t-z}^{1+2s}} \, dz dt 
    \\ & = 2 \e^{1-2s} \int_0^{\ell/\e} \int_t^{\sfrac{\ell}{\e}} \frac{\abs{w(t) - w(z)}^2}{\abs{t-z}^{1+2s}} \, dz dt 
    \\ & \lesssim \e^{1-2s} \int_0^{\sfrac{\ell}{\e}} \left[ \int_t^{t+1} + \int_{t+1}^{\sfrac{\ell}{\e}} \right] \frac{\abs{w(t) - w(z)}^2}{\abs{t-z}^{1+2s}} \, dz dt 
    \\ & = \e^{1-2s} [I_\e + II_\e]. 
\end{align}
Using the decay of the derivative of the optimal profile (see \cref{t:optimal profile}), we have 
\begin{equation}
    I_\e \lesssim \int_0^{\sfrac{\ell}{\e}} \frac{1}{1+ \abs{t}^{2+4s}} \int_t^{t+1} \abs{t-z}^{1-2s} \, dz  dt \lesssim 1. 
\end{equation}
Similarly, we see that
\begin{align}
    II_\e \lesssim \int_0^{\sfrac{\ell}{\e}} (1- w(t))^2\int_{t+1}^{+\infty} \frac{1}{\abs{t-z}^{1+2s}} \, dz  dt \lesssim \left(\int_0^{\sfrac{\ell}{\e}} \frac{1}{1+ \abs{t}^{4s}} \, dt \right) \left(\int_1^{+\infty} \frac{1}{z^{1+2s}} \, dz \right) \lesssim 1 + \e^{4s-1}, 
\end{align}
thus proving \eqref{eq claim 5}. 
\end{proof}

\begin{center}
	{\bf Acknowledgments} 
\end{center}

H.~C. was supported by the Swiss National Science Foundation under Grant
PZ00P2\_202012/1. S.~D. was supported by the Australian Research Council Future Fellowship FT230100333 “New perspectives on nonlocal equations”. M.~F. is funded by the European Union: the European Research Council (ERC), through StG ``MAGNETIC'', project number: 101165368. Views and opinions expressed are however those of the authors only and do not necessarily reflect those of the European Union or the European Research Council. Neither the European Union nor the granting authority can be held responsible for them. M.~F. is a member of the PRIN Project 2022AKNSE4 “Variational and Analytical aspects of Geometric PDEs”. M.~F. acknowledges the INdAM-GNAMPA Project ``Gamma-convergenza di funzionali
geometrici non-locali'', CUP \#E5324001950001\#. M.~I. was partially supported by the SNF grant FLUTURA  “Fluids, Turbulence, Advection” No. 212573. E.~V. was supported by the Australian Laureate Fellowship FL190100081 “Minimal surfaces, free boundaries and partial differential equations”.

\bibliographystyle{plain} 
\bibliography{biblio}

\end{document}